\documentclass[10pt,oneside,reqno]{amsart}
\usepackage[T1]{fontenc}
\usepackage{amsmath,amsthm, mathtools} 
\usepackage{enumerate}
\usepackage{xfrac}          

\usepackage{comment}
\usepackage[backend = biber, style = alphabetic
]{biblatex}
	\renewbibmacro{in:}{}
	\AtEveryBibitem{\clearlist{language}}

\numberwithin{equation}{section}
\theoremstyle{definition}
\newtheorem{Definition}{Definition}[section]
\newtheorem{Example}[Definition]{Example}
\newtheorem{Remark}[Definition]{Remark}

\theoremstyle{plain}
\newtheorem{Theorem}[Definition]{Theorem}

\newtheorem{Proposition}[Definition]{Proposition}

\newtheorem{Lemma}[Definition]{Lemma}

\usepackage{amssymb}        
\usepackage{mathrsfs}       
\usepackage{eucal}          
\usepackage{stmaryrd}       

\usepackage{geometry}

\usepackage{hyperref}

\usepackage[usenames,dvipsnames]{xcolor}
\usepackage{tikz}
\usetikzlibrary{arrows, matrix}
\usepackage{tikz-cd}
\usepackage{subfig}         
\usepackage{arydshln}       

\newcommand{\bDiamond}{\mathbin{\Diamond}}

\makeatletter
\newcommand\bigDiamond{\mathop{\mathpalette\bigDi@mond\relax}}
\newcommand\bigDi@mond[2]{%
	\vcenter{\hbox{\m@th
			\scalebox{\ifx#1\displaystyle 2\else1.2\fi}{$#1\Diamond$}%
	}}%
}
\makeatother

\usepackage{romanbar}

\newcommand{\al}{\alpha}
\newcommand{\be}{\beta}

\newcommand{\la}{\lambda}

\newcommand{\Z}{\mathbb{Z}}

\newcommand{\C}{\mathbb{C}}

\newcommand{\Fgl}{\mathfrak{gl}}
\newcommand{\Fsl}{\mathfrak{sl}}

\newcommand{\Fosp}{\mathfrak{osp}}

\newcommand{\Fg}{\mathfrak{g}}
\newcommand{\Fh}{\mathfrak{h}}

\newcommand{\Fn}{\mathfrak{n}}

\newcommand{\Fp}{\mathfrak{p}}

\newcommand{\FG}{\mathfrak{G}}

\newcommand{\CC}{\mathcal{C}}

\newcommand{\CL}{\mathcal{L}}

\newcommand{\op}{\operatorname}

\DeclareMathOperator{\End}{End}

\DeclareMathOperator{\Hom}{Hom}
\DeclareMathOperator{\Id}{Id}

\DeclareMathOperator{\im}{im}

\DeclareMathOperator{\Span}{Span}

\renewcommand{\hat}{\widehat}
\renewcommand{\tilde}{\widetilde}

\newcommand{\doublecoset}[3]{#1\reflectbox{$/$}#2/#3}

\makeatletter
\def\ifemptyarg#1{%
	\if\relax\detokenize{#1}\relax 
	\expandafter\@firstoftwo
	\else
	\expandafter\@secondoftwo
	\fi}
\makeatother

\makeatletter
\def\ifemptyarg#1{%
		\if\relax\detokenize{#1}\relax
	\expandafter\@firstoftwo
	\else
	\expandafter\@secondoftwo
	\fi}
\makeatother

\newcommand\dual[1]{{#1}^{\ast}}

\title{Diagonal reduction algebra for $\Fosp(1|2)$}
\author{Jonas T. Hartwig \and Dwight Anderson Williams II}
\date{February 26, 2022} 

\address{Department of Mathematics, Iowa State University, Ames, IA-50011, USA}
\email{jth@iastate.edu}
\urladdr{http://jthartwig.net}
\address{MathDwight, The Bronx, New York, NY-10462, USA}
\email{dwight@mathdwight.com}
\urladdr{https://mathdwight.com}
\begin{document}
\maketitle
\begin{abstract}
The problem of providing complete presentations of reduction algebras associated to a pair of Lie algebras $(\mathfrak{G},\mathfrak{g})$ has previously been considered by Khoroshkin and Ogievetsky in the case of the diagonal reduction algebra for $\mathfrak{gl}(n)$. In this paper, we consider the diagonal reduction algebra of the pair of Lie superalgebras $\left(\mathfrak{osp}(1|2) \times \mathfrak{osp}(1|2), \mathfrak{osp}(1|2)\right)$ as a double coset space having an associative diamond product and give a complete presentation in terms of generators and relations. We also provide a PBW basis for this reduction algebra along with Casimir-like elements and a subgroup of automorphisms.
\end{abstract}


\section{Introduction}
\label{sec:introduction}
 Zhelobenko \cite{zhelobenkoRepresentationsReductiveLie1994} developed a localized version of Mickelsson's step algebra 
 \cite{mickelssonStepAlgebrasSemisimple1973} in considering the representation theory of reductive Lie algebras. An overview of the application of Mickelsson-Zhelobenko algebras in determining branching rules for classical Lie algebras and bases of Gelfand-Tsetlin type is found in \cite{molevGelfandTsetlinBases2006}. An important generalization made by Zhelobenko is a representation-free determination of lowering and raising operators associated to a Lie algebra $\Fg$ (reductive in a Lie algebra $\mathfrak{G}$) and an associative algebra $U$ containing the universal enveloping algebra $U(\Fg)$, with the operators satisfying certain dynamical relations. This algebra of lowering and raising operators provides a foundation for describing dynamical Weyl groups, as mentioned in \cite{etingofDynamicalWeylGroups2002}, and yields an alternative route to the developments of \cite{tarasovDifferenceEquationsCompatible2000}. Recent directions of applications can be seen in harmonic analysis, such as in \cite{debieHarmonicTransvectorAlgebra2017}, where the algebra is called the transvector algebra. These class of algebras are also known as symmetry algebras due to their importance in understanding symmetries of extremal systems \cite{zhelobenkoHypersymmetriesExtremalEquations1997} key to mathematical physics, including solutions to the usual Laplace operator, Dirac equation, and Maxwell equations. 

  Moreover, the algebras themselves have been an object of study with much progress made by Khoroshkin, Ogievetsky, et al. \cite{khoroshkinMickelssonAlgebrasZhelobenko2008, khoroshkinRingsFractionsReduction2014}. We follow Khoroshkin and Ogievetsky's naming convention and use the term \emph{reduction algebra} over the aforementioned choices. The general construction of reduction algebras extends to super, quantum, and affine cases \cite{ashtonRmatrixMickelssonAlgebras2015,matsumotoRepresentationsCentrallyExtended2014b, vandenhomberghNoteMickelssonStep1975, zhelobenkoExtremalProjectorsGeneralized1989} and the embedding of the reductive algebra $\Fg$ into the larger associative algebra $U$ gives rise to types of reduction algebras. In \cite{khoroshkinDiagonalReductionAlgebra2017}, a main result is the complete presentation of generators and relations for the diagonal reduction algebra for $\Fgl(n)$.
  Determining complete presentations of diagonal reduction superalgebras connected to Lie superalgebras is an open problem in the study of associative superalgebras and super representation theory. Particularly, we address the diagonal reduction algebra associated with the first Lie superalgebra in the $B(m,n)$ series, \cite{kacLieSuperalgebras1977} that being $\Fosp(1|2)$. Our main result is a complete presentation of the diagonal reduction algebra $Z(\Fosp(1|2) \times \Fosp(1|2), \Fosp(1|2))$.	
  
  In Section 2, we establish the constructions of \cite{zhelobenkoExtremalProjectorsGeneralized1989} and \cite{asherovaProjectionOperatorsSimple1973} in defining the reduction algebra $Z(\mathfrak{G},\Fg)$ and extremal projector associated with a Lie superalgebra $\Fg$, where $\Fg$ is reductive in $\mathfrak{G}$.
   
  Section 3 recalls descriptions of the orthosymplectic Lie superalgebra $\Fosp(1|2)$ before we present the diagonal reduction algebra for $\Fosp(1|2)$, a PBW basis theorem with respect to the diamond product of \cite{khoroshkinMickelssonAlgebrasZhelobenko2008}, and an explicit realization of $Z(\Fosp(1|2) \times \Fosp(1|2), \Fosp(1|2))$ in terms of generators and relations.
  
  Section 4 concludes with applications, including a determination of Casimir-like elements and an infinite subgroup of the group of automorphisms of $Z(\Fosp(1|2) \times \Fosp(1|2), \Fosp(1|2))$. 
  
\section*{Acknowledgments}

The first author was supported by Simons Foundation Collaboration Grant \#637600.
The second author was supported by NSF ECR-EHR Core Research Grant 
\#1920753.

\section{Reduction algebras for Lie superalgebras}\label{sec:RedAlgLSA}

In this section, we model the constructions of \cite{zhelobenkoExtremalProjectorsGeneralized1989} and offer a categorical perspective of the extremal projector of \cite{asherovaProjectionOperatorsSimple1973} and related works.

\subsection{Reduction algebra}
Let $\mathfrak{G}$ (Fraktur G) be a Lie superalgebra and $\Fg\subset\mathfrak{G}$ a Lie subsuperalgebra reductive in $\mathfrak{G}$ with a triangular decomposition $\Fg=\Fg_-\oplus\Fh\oplus\Fg_+$. Let $D$ be a multiplicative subset of $U(\Fh)\setminus\{0\}$, and let $U=D^{-1}U(\mathfrak{G})=D^{-1}U(\Fh)\otimes_{U(\Fh)}U(\mathfrak{G})$ be the localization of $U(\mathfrak{G})$ at the set $D$. 
Let $I=U\Fg_+$ be the left ideal in $U$ generated by $\Fg_+$, employing natural inclusions, $\Fg_{+} \xhookrightarrow{} U$, for example.
Let $N=N_U(I)=\{x\in U\mid Ix\subset I\}=\{x\in U\mid \Fg_+x\subset I\}$ be the normalizer of $I$ in $U$. Note that $D^{-1}U(\Fh)\subset N$ since $xh=hx-[h,x]\in I$ for all $x\in \Fg_+$ and $h\in\Fh$.
The \emph{($D$-localized) reduction algebra} is defined to be $Z=Z(\mathfrak{G},\Fg;D) = Z(U,\Fg)=N/I$. We often write $Z(\mathfrak{G}, \Fg)$ for $Z(\mathfrak{G},\Fg;D)$ when the multiplicative set $D$ is understood.

\subsection{Category \texorpdfstring{$\CC(U,\Fg_+)$}{of locally finite U-modules}}
A left $U$-module $V$ is \emph{locally $\Fg_+$-finite} if $\dim_\C U(\Fg_+)v<\infty$ for all $v\in V$.
Let $\CC=\CC(U,\Fg_+)$ be the category whose objects are locally $\Fg_+$-finite $U$-modules and morphisms are $U$-module homomorphisms. 

\subsection{Invariants and coinvariants}
We consider two covariant functors from $\CC$ to the category of super vector spaces $\mathsf{SVect_\C}$:
\begin{align}
(-)^+\colon\CC &\to\mathsf{SVect_\C}, \quad V^+=\{v\in V\mid \Fg_+v=0\},\\
(-)_-\colon\CC &\to\mathsf{SVect_\C}, \quad V_-=V/\Fg_-V.
\end{align}
The set $V^+$ is the space of \emph{$\Fg_+$-invariants} of $V$, and $V_-$ is the space of \emph{$\Fg_-$-coinvariants} of $V$. The elements of $V^+$ are also called \emph{primitive} (or \emph{singular}) \emph{vectors}.
Associated to these functors are natural transformations to and from the forgetful functor $\mathsf{F}\colon \CC\to\mathrm{SVect}_\C$, given by inclusion, respectively, canonical projection:
\begin{align}
\iota&\colon(-)^+\Rightarrow \mathsf{F}, \quad \iota_V\colon V^+\hookrightarrow V,\\
\pi&\colon\mathsf{F}\Rightarrow (-)_-, \quad \pi_V\colon V\twoheadrightarrow V_-.
\end{align}
Define $\mathsf{Q}$ the be the vertical composition of the natural transformations above:
\begin{equation}
\mathsf{Q}\colon(-)^+\Rightarrow (-)_-,\quad \mathsf{Q}_V=\pi_V\circ\iota_V.
\end{equation}

\begin{equation}
\begin{tikzcd}[row sep=huge]
    \CC
     \arrow[r, bend left=65, "(-)^+"{name=I}, end anchor={[xshift=-2.5ex]north}]
     \arrow[r, "F"{inner sep=0,fill=white,anchor=center,name=F}]
     \arrow[r, bend right=65, "(-)_-"{name=C, swap}, end anchor={[xshift=-2.5ex]south}]
     \arrow[from=I.south-|F,to=F,Rightarrow,shorten=2pt,"\iota"] 
     \arrow[from=F,to=C.north-|F,Rightarrow,shorten=2pt,"\pi"] &
   \mathsf{SVect}
\end{tikzcd} 
\end{equation}

Explicitly, for each object $V$ of $\CC$, we have a map of super vector spaces $\mathsf{Q}_V\colon V^+\to V_-$ given by $\mathsf{Q}_V(v)=v+\Fg_-V$ for all $v\in V^+$.

\subsection{Extremal projectors}
\label{sec:extremal-projectors}

\begin{Definition}
We say that the category $\CC$ has an \emph{extremal projector} if the natural transformation $\mathsf{Q}$ is invertible. In this case we denote the inverse by $\mathsf{P}\colon(-)_-\Rightarrow (-)^+$.
\end{Definition}

\begin{Proposition} \label{prp:extremal}
The following statements are equivalent:
\begin{enumerate}[{\rm (a)}]
\item $\CC$ has an extremal projector.
\item There exists a natural endomorphism $P$ of the forgetful functor $\mathsf{F}\colon\CC\to\mathsf{SVect}_\C$ such that for all objects $V$ in $\CC$:
\begin{enumerate}[{\rm (i)}]
\item $P_V\circ \iota_V =\iota_V$, 
\item $\pi_V\circ P_V = \pi_V$, 
\item $\im P_V\subset\im \iota_V$,
\item $\ker \pi_V\subset \ker P_V$,
\end{enumerate}
\item
For every $U$-module $V$ in $\CC$ there exists a map of super vector spaces $P_V \colon V\to V$ such that, for every morphism $f\colon V\to W$ in $\CC$ we have $f\circ P_V = P_W\circ f$, and for every object $V$ in $\CC$ and 
every $v\in V$:
\begin{enumerate}[{\rm (i')}]
\item if $\Fg_+v=0$, then $P_V(v)=v$,
\item $P_V(v)\in v+\Fg_-V$,
\item $\Fg_+P_V(v)=0$,
\item $P_V(\Fg_-v)=0$.
\end{enumerate}
\end{enumerate}
Furthermore, if these statements hold, then for all objects $V$ in $\CC$:
\begin{equation}
P_V^2=P_V.
\end{equation}
\end{Proposition}

\begin{proof}
Statements (b) and (c) are equivalent by definition of the involved notions.
(a)$\Rightarrow$(b):
Suppose $\CC$ has an extremal projector and let
 $\mathsf{P}\colon(-)_-\Rightarrow(-)^+$ be the inverse of $\mathsf{Q}$.
For every object $V$ of $\CC$, define $P_V=\iota_V\circ\mathsf{P}_V\circ\pi_V$. Then $P$ is the vertical composition of the natural transformations $\iota$, $\mathsf{P}$, and $\pi$, hence is a natural endomorphism of $\mathsf{F}$. We check the properties: 

(i) $P_V\circ\iota_V=\iota_V\circ\mathsf{P}_V\circ\pi_V\circ\iota_V=\iota_V\circ\mathsf{P}_V\circ\mathsf{Q}_V=\iota_V$.
(ii) $\pi_V\circ P_V=\pi_V\circ\iota_V\circ\mathsf{P}_V\circ\pi_V=\mathsf{Q}_V\circ\mathsf{P}_V\circ\pi_V=\pi_V$.
(iii) and (iv) are immediate by definiton of $P_V$.

(b)$\Rightarrow$(a): Suppose $P$ is a natural endomorphism of $\mathsf{F}$ satisfying (i)-(iv). Let $V$ be an object in $\CC$. By (iii) and (iv), $P_V$ induces a unique map $\mathsf{P}_V\colon V_-\to V^+$ such that $\iota_V\circ\mathsf{P}_V\circ\pi_V=P_V$. By uniqueness of $\mathsf{P}_V$ and the naturality of $P_V$, $\iota_V$, and $\pi_V$, the maps $\mathsf{P}_V$ define a natural transformation $\mathsf{P}\colon (-)_-\Rightarrow (-)^+$. It remains to show that $\mathsf{P}$ is the inverse of $\mathsf{Q}$.

We have $\iota_V\circ\mathsf{P}_V\circ\mathsf{Q}_V = \iota_V\circ\mathsf{P}_V\circ\pi_V\circ\iota_V = P_V\circ\iota_V=\iota_V$ by (i). Since $\iota_V$ is monic, this implies $\mathsf{P}_V\circ\mathsf{Q}_V=\Id_{V^+}$.
Similarly, $\mathsf{Q}_V\circ\mathsf{P}_V\circ\pi_V= \pi_V\circ\iota_V\circ\mathsf{P}_V\circ\pi_V = \pi_V\circ P_V=\pi_V$ by (ii). Since $\pi_V$ is epic, this implies $\mathsf{Q}_V\circ\mathsf{P}_V=\Id_{V_-}$.

Finally,
$P_V^2=(\iota_V\circ\mathsf{P}_V\circ\pi_V)^2=\iota_V\circ\mathsf{P}_V\circ\mathsf{Q}_V\circ\mathsf{P}_V\circ\pi_V=\iota_V\circ\mathsf{P}_V\circ\pi_V=P_V$.
\end{proof}

\begin{Definition}
If $V$ is an object of $\CC$, we  call $P_V$ the \emph{extremal projector at $V$}.
\end{Definition}

\subsection{Universal highest weight module}
The \emph{universal $\Fg_+$-highest weight $U$-module} is
\begin{equation}
 M=U/I.
\end{equation}
This is a left $U$-module generated by the vector $\mathbf{1}=1_U+I\in M$ which satisfies $\Fg_+\mathbf{1}=0$.
Since $IN\subset I$, the space $M$ is a right $N$-module. Furthermore $MI=0$. Thus $M$ is a right $Z$-module. Together, this makes $M$ a $(U,Z)$-bimodule.

Let $V$ be any left $U$-module.
Restricting the action to $N$ and using that $I V^+=0$ we may regard $V^+$ as a $Z$-module. On the other hand, since $M$ is a $(U,Z)$-bimodule, $\Hom_U(M,V)$ is a left $Z$-module.
\begin{Lemma} \label{lem:representable}
For any left $U$-module $V$, there is a natural isomorphism of left $Z$-modules
\begin{equation}\label{eq:representable}
\psi\colon V^+ \cong \Hom_U(M, V).
\end{equation}
\end{Lemma}

\begin{proof}
Send $v\in V^+$ to the unique left $U$-module map $\psi_v\colon M\to V$ determined by requiring $\psi_v(\mathbf{1})=v$. The inverse sends $\psi\colon M\to V$ to $\psi(\mathbf{1})$.
\end{proof}

Of particular importance is the case $V=M$, which gives another realization of the reduction algebra $Z$.

\begin{Lemma}\label{lem:Z-realizations}
We have the following two descriptions of $Z$:
\begin{enumerate}[{\rm (i)}]
\item $Z=M^+$.
\item There is a natural isomorphism of associative superalgebras
\begin{equation}
Z\cong \End_U(M)^{\mathrm{op}}.
\end{equation}
\end{enumerate}
\end{Lemma}

\begin{proof}
(i) A coset $u+I\in M$ is in $M^+$ if and only if $Iu\subset I$ which by definition means $u\in N$.

(ii) 
Let $\psi\colon M^+\to\End_U(V)$ be the super vector space isomorphism  \eqref{eq:representable} for the special case of $V=M$. Let  $X,Y\in Z$, say $X=x+I$, $Y=y+I$. Then $\psi_{XY}(\mathbf{1})=XY$ while $(\psi_Y\circ\psi_X)(\mathbf{1}) = \psi_Y(X)=\psi_Y(x+I)$. Since $\psi_Y$ is a left $U$-module endomorphism, we have
$\psi_Y(x+I)=x \psi_Y(1_{U}+I)=xY=xy+I=XY$.
\end{proof}

\begin{Remark}
By tensor-hom adjunction we have $(M\otimes_Z -) \dashv \Hom_U(M,-)$. More precisely,
\begin{equation}
\Hom_U(M\otimes_Z X,V)\cong\Hom_Z(X,\Hom_U(M,V))\cong \Hom_Z(X,V^+)
\end{equation}
for any left $U$-module $V$ and left $Z$-module $X$.
\end{Remark}

\begin{Lemma}\label{lem:M-is-projective}
If $\CC$ has an extremal projector, then $M$ is a projective object of $\CC$.
\end{Lemma}

\begin{proof}
Consider a diagram in $\CC$
\[
\begin{tikzcd}
& M \arrow[d,"f"] & \\
X\arrow[r,"\phi"] & Y \arrow[r] & 0
\end{tikzcd}
\]
where the row is exact. We will show that $f$ can be lifted to a map $g\colon M\to X$ such that the diagram commutes.
Let $y_0$ be the image of $\mathbf{1}$.
Since $\phi$ is surjective there is an $x\in X$ such that $\phi(x)=y_0$. Define $x_0=P_X(x)$ where $P_X$ is the extremal projector at $X$. Then
\[
\phi(x_0) = \phi\circ P_X(x) = P_Y\circ\phi(x)=P_Y(y_0)=P_Y(f(\mathbf{1}))=f(P_M(\mathbf{1}))=f(\mathbf{1})=y_0.
\]
(Here we used that $P_M(\boldsymbol{1})=\boldsymbol{1}$ which follows from property (i') of Proposition \ref{prp:extremal}.)
Moreover $\Fg_+x_0 = \Fg_+P_Xx=0$. So there is a unique $U$-module map $g\colon M\to X$ satisfying $g(\mathbf{1})=x_0$. Clearly $\psi\circ g= f$ since both sides coincide on the generator $\mathbf{1}$ of $M$.
\end{proof}

\subsection{Diamond product on the double coset space}
The space of coinvariants $M_-=M/\Fg_-M$ is naturally identified with the double coset space $\doublecoset{\Fg_-U}{U}{U\Fg_+}=U/\Romanbar{II}$ where $\Romanbar{II}=U\Fg_+ + \Fg_-U$. Explicitly, $(u+I)+\Fg_-M\mapsto u+\Romanbar{II}$ for $u\in U$.

Suppose $\CC$ has an extremal projector. Then we may define a product $\bDiamond$ on $M_-$ by requiring that the super vector space isomorphism $\mathsf{P}_M\colon M_-\to M^+=Z$ is an algebra isomorphism. That is, define for all $x,y\in M_-$:
\[
x \bDiamond y = \mathsf{Q}_M\big(\mathsf{P}_M(x)\mathsf{P}_M(y)\big).
\]
Since $M$ is a right $Z$-module and $\Fg_-M$ is a right $Z$-submodule, the quotient space $M_-$ is a right $Z$-module. The following describes the diamond product in terms of the extremal projector and this right action of $Z$ on $M_-$, remembering that $M^+=Z$.

\begin{Lemma}[{\cite{khoroshkinMickelssonAlgebrasZhelobenko2008}}]
For $x,y\in M_-$, we have 
\begin{equation}
x \bDiamond y = x\mathsf{P}_M(y)
\end{equation}
\end{Lemma}

\begin{proof}
Let $z\in Z$. Since $M$ is a $(U,Z)$-bimodule, the right action of $z$ on $M$ is a $U$-module endomorphism. By the functoriality of $(-)^+$ and $(-)_-$ this induces super vector space endomorphisms $z^+\colon M^+\to M^+$ and $z_-\colon M_-\to M_-$. Explicitly, $z^+(x+I)=xz+I$ for $x+I\in M^+$ and $z_-(x+\Romanbar{II})=xz'+\Romanbar{II}$ for $x+\Romanbar{II}\in M_-$, where $z'\in N$ is any representative of $z=z'+I\in Z$. By the naturality of $\mathsf{Q}$, we have a commutative diagram
\[\begin{tikzcd}
M^+ \arrow[d,"z^+"] \arrow[r,"\mathsf{Q}_M"]  & M_- \arrow[d,"z_-"]\\
M^+ \arrow[r,"\mathsf{Q}_M"] & M_-
\end{tikzcd}\]

In other words, $\mathsf{Q}_M\colon M^+\to M_-$ is a map of right $Z$-modules. With this in mind, and that $Z=M^+$, we have for all $x,y\in M_-$:
\[
x \bDiamond y 
= \mathsf{Q}_M\big(\mathsf{P}_M(x)\mathsf{P}_M(y)\big)
= \mathsf{Q}_M\big(\mathsf{P}_M(x)\big)\mathsf{P}_M(y)=
x\mathsf{P}_M(y).
\]
\end{proof}

\subsection{Generators for the algebra \texorpdfstring{$Z$}{Z}}

Since $\Fg$ is reductive in $\mathfrak{G}$, there exists a $\Fg$-module complement $\Fp$ of $\Fg$ in $\mathfrak{G}$, i.e. $\mathfrak{G}=\Fg\oplus\Fp$ as $\Fg$-modules. Consider the composition of maps

\begin{equation}\label{eq:p-to-Z}
\Fp\hookrightarrow \mathfrak{G} \hookrightarrow U\twoheadrightarrow U/\Romanbar{II}=M_- \stackrel{\mathsf{P}_M}{\longrightarrow} M^+ \cong Z
\end{equation}

Recall that $D^{-1}U(\Fh)\subset N$, hence $Z$ is a $D^{-1}U(\Fh)$-ring (meaning there is an algebra map $D^{-1}U(\Fh)\to Z$ whose image is not necessarily contained in the center of $Z$).

\begin{Lemma}\label{lem:generators}
The image of $\Fp$ under the map \eqref{eq:p-to-Z} generates $Z$ as a $D^{-1}U(\Fh)$-ring. Hence the image of $\Fp$ in $U/\emph{\Romanbar{II}}$ generates $U/\emph{\Romanbar{II}}$ as a $D^{-1}U(\Fh)$-ring with respect to the diamond product.
\end{Lemma}

\begin{proof}
By the PBW theorem applied to the decomposition $\mathfrak{G}=\Fg_-\oplus \Fh\oplus \Fp\oplus \Fg_+$, we have
\[
U \cong U(\Fg_-)\otimes D^{-1}U(\Fh) \otimes U(\Fp) \otimes U(\Fg_+).
\]
Writing $U(\Fg_+)=\C \oplus U(\Fg_+)\Fg_+$ and using that $U(\Fg_+)\Fg_+\subset I$ we see that 
 $U(\Fg_-)\otimes D^{-1}U(\Fh) \otimes U(\Fp)$ still maps onto $U/I$. On the other hand, since $P_M\Fg_-=0$, the subspace $D^{-1}U(\Fh) \otimes U(\Fp)$ of $U$ still maps onto $M^+$. This proves the claim.
\end{proof}

\subsection{Irreducibility of the \texorpdfstring{$Z$-modules $V^+$}{space of primitive vectors as Z-modules}}\label{sec:IrredZmod}

\begin{Proposition}
Let $V$ be an object of the category $\CC$.
If $v\in V^+$ generates $V$ as a $U$-module then $v$ generates $V^+$ as a $Z$-module. In particular, if $V$ is a simple $U$-module then $V^+$ is a simple $Z$-module.
\end{Proposition}

\begin{proof}
Since $V$ is generated by $v$, the map $\psi_v$ is surjective. Let $K$ be the kernel of $\psi_v$.
The short exact sequence of $U$-modules
\[ 0 \to K \to M \to V \to 0 \]
gives rise to a long exact sequence
\[ 0\to \Hom_U(M,K)\to\Hom_U(M,M)\to\Hom_U(M,V)\to\op{Ext}^1_U(M,K)\to \cdots \]
Since $M$ is projective by Lemma \ref{lem:M-is-projective}, we have $\op{Ext}^1_U(M,K)=0$. Therefore the map $\Hom_U(M,M)\to\Hom_U(M,V)$ given by $z\mapsto \psi_v\circ z$ is surjective. Note that $\Hom_U(M,M)=\End_U(M)\cong Z^{\mathrm{op}}$ by Lemma \ref{lem:Z-realizations} and $\Hom_U(M,V)\cong V^+$ by Lemma \ref{lem:representable}. Under these identifications the map in question is simply the map $Z\to V^+$ given by $z\mapsto zv$. To say that this is surjective is equivalent to saying that $v$ generates $V^+$ as a $Z$-module.
\end{proof}

\section{Diagonal reduction algebra of \texorpdfstring{$\Fosp(1|2)$}{osp(1|2)}}
Here we initiate a concrete exploration of the algebra $Z = Z(\FG, \Fg; D)$, where $\FG$ is the Lie superalgebra $\Fosp(1|2) \times \Fosp(1|2)$ and $\Fg$ is the image in $\FG$ of $\Fosp(1|2)$ under the diagonal embedding. The denominator set $D$ of Section \ref{sec:RedAlgLSA} is the multiplicative set generated by $\{ H - n \mid n \in \Z \}$, with $H = (h, h)$ in $\Fosp(1|2) \times \Fosp(1|2)$, chosen to allow for the existence of the extremal projector for $\Fg$ (see \cite{tolstoyExtremalProjectionsReductive1985, berezinGroupGrassmannStructure1981}) described in Section \ref{sec:extremal-projector-for-osp}. We investigate the structure of the superalgebra $Z$ by providing generators and relations and determining a PBW basis, with respect to the diamond ($\bDiamond$) product, to justify a complete presentation. All (super) vector spaces and (super)algebras will be considered as objects over the field of complex numbers, unless otherwise stated. Our basic references for Lie superalgebras are
\cite{frappatDictionaryLieAlgebras2000}, \cite{mussonLieSuperalgebrasEnveloping2012}, \cite{chengDualitiesRepresentationsLie2012}.

\subsection{The Lie superalgebra \texorpdfstring{$\Fosp(1|2)$}{osp(1|2)}}
\subsubsection{Definition}
The Type II basic classical Lie superalgebra $\Fosp(1|2)$ is the super vector space spanned by $\{x_{-2\al}, h, x_{2\al}; x_{-\al}, x_{\al}\}$ of dimension (3|2) (super dimension 1) preserving an even, nondegenerate, supersymmetric bilinear form on a super vector space of dimension (1|2) (superdimension -1).
Equivalently, let $\Fgl(1|2)$ be the set of all linear transformations on $\C^{1|2}$ expressed as matrices with respect to a standard basis $\{v_{0}; v_{1}, v_{2} \}$ of $\C^{1|2}$, 
	with even vector \[v_{0} = \begin{bmatrix} 1 \\ 0 \\ 0 \end{bmatrix}\] 
	and odd vectors
	
	\[v_{1} = \begin{bmatrix} 0 \\ 1 \\ 0 \end{bmatrix}, ~ v_{2} = \begin{bmatrix} 0 \\ 0 \\ 1 \end{bmatrix}. \] 
\textbf{}Elements of $\Fgl(1|2)$ are block matrices $\begin{bmatrix} a & r \\ c & A \end{bmatrix}$; here, $a$ is a scalar, $r$ is a row vector, $c$ is a column vector, and $A$ is a $2\times2$ square matrix. The orthosymplectic Lie superalgebra $\Fosp(1|2)$ is the Lie subsuperalgebra of $\Fgl(1|2)$ in which members subscribe to the conditions 
$a = 0$, $r = (r_{1},r_{2})$, $c = \begin{pmatrix} r_{2} \\ -r_{1} \end{pmatrix}$, and $A$ is an element of $\Fsl(2)$.
For the sake of completeness, we give the supercommutator relations on 
\[\Fosp(1|2) = \Fosp(1|2)_{\bar{0}} \oplus \Fosp(1|2)_{\bar{1}} = \left( \C x_{-2\al} \oplus \C h \oplus \C x_{2\al} \right) \oplus \left( \C x_{-\al} \oplus \C x_{\al} \right), \]
where,
\begin{gather}
	x_{-2\al} =\begin{bmatrix}0&0&0\\0&0&1\\0&0&0\end{bmatrix},\;
	x_{-\al}=\begin{bmatrix}0&0&1\\1&0&0\\0&0&0\end{bmatrix},\;
	h=\begin{bmatrix}0&0&0\\0&1&0\\0&0&-1\end{bmatrix},\;
	x_\al=\begin{bmatrix}0&1&0\\0&0&0\\-1&0&0\end{bmatrix},\;
	x_{2\al}=\begin{bmatrix}0&0&0\\0&0&0\\0&1&0\end{bmatrix},
\end{gather}
which are
  \begin{align}
	[h, x_{k\al}] &= -kx_{k\al}, \quad k\in\{\pm 1,\pm 2\} \label{comm-hx} \\
	[x_{\al}, x_{\al}] &= -2x_{2\al}, \label{comm-xalxal}\\ 
	[x_{-\al}, x_{-\al}] &= 2x_{-2\al} \label{comm-xnalxnal} \\
	[x_{\al}, x_{-2\al}] &= x_{-\al} \label{comm-xalxn2al}\\ 
	[x_{-\al}, x_{2\al}] &= x_{\al} \label{comm-xnalx2al}\\
	[x_{\al}, x_{-\al}] &= h \label{comm-xalxnal}\\
	[x_{-2\al}, x_{2\al}] &= h \label{comm-xn2alx2al}\\
	[x_{\al}, x_{2\al}] &= 0 \label{comm-xalx2al}\\ 
	[x_{-\al}, x_{-2\al}] &= 0 \label{comm-xnalxn2al}.
	\end{align}
%
The even part $\Fosp(1|2)_{\bar{0}}$ of $\Fosp(1|2)$ is a Lie algebra isomorphic to $\Fsl(2)$, and the odd part  $\Fosp(1|2)_{\bar{1}}$ is isomorphic to the natural $\Fsl(2)$-module $\C^{2}$.   The Cartan subalgebra $\Fh$ of $\Fosp(1|2)$ is equal to the Cartan subalgebra of $\Fosp(1|2)_{\bar{0}}$, which is $\C h$.
In Figure \ref{fig:table1} we summarize some useful iterated brackets in $\Fosp(1|2)$ that will be used in connection with the extremal projector.

\setlength{\tabcolsep}{2pt}
\begin{figure}	
\centering
	\begin{tabular}{ccccc|c|ccccc}
		$[\,\cdot\,,x_{-\al}]^5$ &$[\,\cdot\,,x_{-\al}]^4$ &$[\,\cdot\,,x_{-\al}]^3$ &$[\,\cdot\,,x_{-\al}]^2$ &$[\,\cdot\,,x_{-\al}]$ &  &
		$[x_\al,\,\cdot\,]$ &
		$[x_\al,\,\cdot\,]^2$ &
		$[x_\al,\,\cdot\,]^3$ &$[x_\al,\,\cdot\,]^4$ &$[x_\al,\,\cdot\,]^5$ \\
		\hline
		& & & & $0$ & $x_{-2\al}$ & $x_{-\al}$ & $h$ & $x_\al$ & $-2x_{2\al}$ & $0$ \\
		& & & $0$ & $2x_{-2\al}$ & $x_{-\al}$ & $h$ & $x_\al$ & $-2x_{2\al}$ & $0$ &\\
		& & $0$ & $2x_{-2\al}$ & $x_{-\al}$ & $h$ & $x_\al$ & $-2x_{2\al}$ & $0$ & &\\
		& $0$ & $2x_{-2\al}$ & $x_{-\al}$ & $h$ & $x_\al$ & $-2x_{2\al}$ & $0$ & & &\\
		$0$ & $-2x_{-2\al}$ & $-x_{-\al}$ & $-h$ & $-x_\al$ & $x_{2\al}$ & $0$ & & & &
	\end{tabular}
\caption{A table of useful right and left adjoint actions of $\mathfrak{osp}(1|2)$ on itself.}
\label{fig:table1}
\end{figure}

\subsubsection{Odd roots of \texorpdfstring{$\Fosp(1|2)$}{osp(1|2)}}

The $BC_1$ root system $\Phi$ of $\Fosp(1|2)$ is given by the union of even (bosonic) roots $\Phi_{0} = \{\pm2\delta_{1} \}$ with odd (fermionic) roots $\Phi_{1} = \{\pm\delta_{1}\}$. Moreover, we choose a non-standard set of positive roots $\Phi^{+} = \{ -\delta_{1}, -2\delta_{1} \}$ associated to the base $\Pi = \{ -\delta_{1} \}$ for $\Phi$. Positive even (respectively, odd) roots, as well as their negative counterparts, are defined with the appropriate intersection with $\Phi_{0}$ (respectively, $\Phi_{1}$).  In particular, $\Fosp(1|2)$ has a lone positive odd root vector $x_{\al}$, for $\al = -\delta_{1}$ identified with $-1$.

It follows that $\Fosp(1|2)$ has a triangular decomposition: \[\Fosp(1|2) = \Fn_{-} \oplus \Fh \oplus \Fn_{+}, \text{ with } \Fn_{\pm} = \C x_{\pm 2\al} \oplus \C x_{\pm \al}.\]  

The Killing form on the Lie superalgebra $\Fosp(1|2)$ is non-degenerate, and it induces a bilinear form on $\dual{\Fh}$ such that $(\alpha, \alpha) = 1$.

\subsubsection{Diagonal embedding}

Throughout the rest of this section, $\Fg$ is a reductive embedding of $\Fosp(1|2)$ into $\mathfrak{G} = \Fosp(1|2) \times \Fosp(1|2)$ as the image of $\delta\colon  \Fosp(1|2) \rightarrow \mathfrak{G}$, $~x \mapsto (x,x)$, for all $x$ in $\Fosp(1|2)$.

To heed the remarks of the previous section, we provide a linear complement of $\Fg$: Namely,  \[ \Fp = \{(x,-x) \mid x \in \Fosp(1|2n)\}, \] 
which is the image of $\delta_{-}\colon  \Fosp(1|2) \rightarrow \mathfrak{G}$, $~x \mapsto (x,-x)$, for all $x$ in $\Fosp(1|2)$.
Then $\mathfrak{G} = \Fg \oplus \Fp$ as $\Fg$-modules. 
Even more, the map $\Fg \rightarrow \Fp$ sending $(x,x)$ to $(x,-x)$, for all $x$ in $\Fosp(1|2)$, is an isomorphism of $\Fg$-modules.

A root vector $x_{\be}$ of $\Fosp(1|2)$ will be identified as $X_{\be}$, respectively, $\tilde{x}_{\be}$, for its image in $\Fg \subset \mathfrak{G}$, respectively, in $\Fp \subset \mathfrak{G}$. That is, $X_{\be} = (x_{\be},x_{\be})$ and $\tilde x_{\be} = (x_{\be},- x_{\be})$. We also put $H$ for $(h,h)$ and $\tilde h = (h,-h)$.
These identifications remain in $\mathfrak{G} \hookrightarrow{} U$ after we fix the multiplicative set $D$ generated by products of factors $(H-m)^n$, where $m$ is an integer and $n$ is a natural number. 

Again for the sake of lucidity:  
The left ideal $I$ in $U$ is $U(\C X_{\al} + \C X_{2\al})$ and we write $\Romanbar{II}$ for the sum $(\C X_{-\al} + \C X_{-2\al})U + U(\C X_{\al} + \C X_{2\al})$  of subspaces.

The algebra $N_{U}(U(\C X_{\al} + \C X_{2\al}))/U(\C X_{\al} + \C X_{2\al})$, denoted here as $Z(\mathfrak{G},\Fg)$, is the diagonal reduction algebra of $\Fosp(1|2)$ associated to the embedding  $\delta\colon  \Fosp(1|2) \subset \Fosp(1|2) \times \Fosp(1|2)$.

\subsubsection{Universal enveloping algebra and PBW basis}

An ordered basis for $\mathfrak{G}$ is the set 
\begin{gather} \{X_{-2\al}, X_{-\al}, H, \tilde{x}_{-2\al}, \tilde{x}_{-\al}, \tilde{h}, \tilde{x}_{\al}, \tilde{x}_{2\al}, X_{\al}, X_{2\al} \} \label{frakG-basis}.
	\end{gather}

The choice of ordered basis for $\mathfrak{G}$ is compatible with $\Romanbar{II}$.

A direct calculation shows that if $[x,y] = z$, then 
\begin{equation} \label{eq:gp-p}
	[X, \tilde{y}] = \tilde{z},\qquad
	[\tilde{x}, \tilde{y}] = Z,\qquad
	[X,Y] = Z,
\end{equation}
	for all root vectors $x$ and $y$ in $\Fosp(1|2)$. 

Considering the PBW basis theorem for Lie superalgebras with regards to \eqref{frakG-basis} gives a basis for the universal enveloping algebra $U(\mathfrak{G})$.
That is, $U(\mathfrak{G})$ is spanned by a linearly independent set of monomials  where each monomial is of the form
\begin{gather}
X_{-2\al}^a X_{-\al}^b H^c \tilde{x}_{-2\al}^p \tilde{x}_{-\al}^q \tilde{h}^r \tilde{x}_{\al}^s \tilde{x}_{2\al}^t X_{\al}^d X_{2\al}^e \end{gather}

with the exponents $b, d, q, s$ no greater than 1 and  $a, c, e, p, r, t$ any natural number.

Thus, $U$ is a free left (right) $D^{-1}U(\Fh)$-ring.

\subsubsection{Anti-automorphism} \label{sec:anti-automorphism}
  In the aid of computation, we define a Lie superalgebra anti-automorphism of $\Fosp(1|2)$ given by:
   \begin{gather}
  	\theta(x_{\pm \al}) =\sqrt{-1} x_{\mp\al},\qquad \theta(x_{\pm 2\al}) = -x_{\mp 2\al},\qquad \theta(h)=h. 
  \end{gather}
We have
$\theta([x,y])=(-1)^{|x||y|}[\theta(y),\theta(x)]$, and $\theta$ is a map of super vector spaces. Thus, it extends to a superalgebra anti-homomorphism \[\Theta\colon  U \rightarrow U, \enskip \Theta(xy) =(-1)^{|x||y|} \Theta(y)\Theta(x). \]
Furthermore, $\Theta(\Romanbar{II})\subseteq\Romanbar{II}$ and hence $\Theta$ induces a linear endomorphism on $U/\Romanbar{II}$. In fact, $\Theta$ is an anti-automorphism of $U/ \Romanbar{II}$ with respect to the diamond product.

\subsection{Extremal projector for \texorpdfstring{$\Fosp(1|2)$}{osp(1|2)}}
\label{sec:extremal-projector-for-osp}
Refer to \cite{tolstoyExtremalProjectorsContragredient2011} for a general background on extremal projectors associated with various algebraic objects. 

For the convenience of the reader, we deduce the expression for the extremal projector of $\Fosp(1|2)$ using the fixed notation of the current section.

The Taylor extension $TU$ of $U$ can be defined as the projective limit
\begin{equation}
TU=\lim_{\longleftarrow} \frac{U}{\Fg_-^nU+U\Fg_+^n}.
\end{equation}
Define $P\in TU$ by
\begin{equation}
P=\sum_{n=0}^\infty \varphi_n(h)x_{-\al}^nx_\al^n
\end{equation}
where $\varphi_0(h)=1$ and for $n>0$, $\varphi_n(h)\in\C(h)$ are rational functions to be determined.
Introduce $\kappa_n(h)\in \C[h]$ by requiring in $U\big(\Fosp(1|2)\big)$:
\begin{equation}
x_\al x_{-\al}^n-(-1)^nx_{-\al}^nx_\al=[x_\al, x_{-\al}^n] = \kappa_n(h)x_{-\al}^{n-1},\quad n\ge 1;\qquad \kappa_0(h)=0.
\end{equation}
Explicitly, by induction and super Lebniz rule,
\begin{equation}
\kappa_n(h) = \sum_{k=0}^{n-1}(-1)^k(h-k)=
\begin{cases}
n/2,&\text{$n$ even,}\\
h-(n-1)/2,&\text{$n$ odd}.
\end{cases}
\end{equation}
We have
\begin{align*}
x_\al P 
&=\sum_{n=0}^\infty \varphi_n(h+1)x_\al x_{-\al}^nx_\al^n \\
&=\sum_{n=0}^\infty \Big(\varphi_n(h+1)(-1)^nx_{-\al}^nx_\al^{n+1} + \varphi_n(h+1)\kappa_n(h)x_{-\al}^{n-1}x_\al^n\Big) \\
&=\sum_{n=0}^\infty \Big((-1)^n\varphi_n(h+1) + \varphi_{n+1}(h+1)\kappa_{n+1}(h)\Big)x_{-\al}^n x_\al^{n+1}.
\end{align*}
Thus we see that $x_\al P=0$ if and only if
\begin{equation}
\varphi_n(h)=\frac{(-1)^n}{\kappa_n(h-1)}\varphi_{n-1}(h),\quad n\ge 1,
\end{equation}
or equivalently,
\begin{equation}
\varphi_{2n-1}(h)=\frac{-1}{h-n}\varphi_{2n-2}(h),\qquad
\varphi_{2n}(h)=\frac{1}{n}\varphi_{2n-1}(h),\qquad n\ge 1.
\end{equation}
Together with the initial condition $\varphi_0(h)=1$ this determines the rational functions $\varphi_n(h)$ uniquely.
The first few values of $\varphi_n(h)$ are
\begin{equation}
\varphi_0(h)=1,\quad
\varphi_1(h)=\varphi_2(h)=\frac{-1}{h-1},\quad
\varphi_3(h)=\frac{1}{(h-2)(h-1)},\quad
\varphi_4(h)=\frac{1}{2}\varphi_3(h).
\end{equation}
We see that the minimal multiplicative set $D$ satisfying the Ore conditions and such that $\varphi_n(h)\in D^{-1}U(\Fh)$ is the multiplicative submonoid of $U(\Fh)\setminus\{0\}$ generated by $\{(h-n)\mid n \in \Z \}$.

\subsection{Generators}

The diamond product on $U/\Romanbar{II}$ is defined as follows. For $u\in U$, let $\bar u=u+\Romanbar{II}\in U/\Romanbar{II}$. For all $u,v\in U$, define
\begin{equation}
	\bar u\bDiamond \bar v = uPv +\Romanbar{II}.\end{equation}
This is well-defined because $\Romanbar{II}P = 0 = P\Romanbar{II}$. This may also be expressed as follows:

\begin{equation}\label{eq:diamond-compute}
	\begin{aligned}
		\bar u\bDiamond\bar v 
		&= uv \\
		&+ [u,X_{-\al}]\cdot \varphi_1(H+1)\cdot [X_\al,v] \\
		&+ [[u,X_{-\al}],X_{-\al}]\cdot\varphi_2(H+2)\cdot[X_\al,[X_\al,v]] \\
		&+ \cdots + \Romanbar{II}.
	\end{aligned}
\end{equation}

\begin{Proposition}
	The algebra $ Z\left(\Fosp(1|2) \times \Fosp(1|2) , \Fosp(1|2)\right)$ is generated as a $D^{-1}U(\Fh)$-ring by the following elements:
	\begin{align*}
		P\tilde x_{2\al} +I &= \tilde x_{2\al} + I \\
		P\tilde x_\al +I&= \tilde x_\al -2 \varphi_1(H) X_{-\al} \tilde x_{2\al}+ I\\
		P\tilde{h}+I &= \tilde h + \varphi_1(H) X_{-\al} \tilde x_\al-2 \varphi_2(H) X_{-\al}^2 \tilde x_{2\al}+I\\
		P\tilde x_{-\al} +I&=\tilde x_{-\al} + \varphi_1(H)X_{-\al}\tilde h+\varphi_2(H)X_{-\al}^2\tilde x_\al-2\varphi_3(H)X_{-\al}^3\tilde x_{2\al}+I\\
		P\tilde x_{-2\al} +I&= \tilde x_{-2\al}+\varphi_1(H)X_{-\al}\tilde x_{-\al}+\varphi_2(H)X_{-\al}^2\tilde h+\varphi_3(H)X_{-\al}^3\tilde x_\al -2\varphi_4(H)X_{-\al}^4\tilde x_{2\al}+I
	\end{align*}
\end{Proposition}

\begin{proof}
Applying Lemma \ref{lem:generators}, $\big\{P(\tilde u + I)\mid \tilde u\in\{\tilde x_{-2\al}, \tilde x_{-\al}, \tilde h, \tilde x_{\al}, \tilde x_{2\al}\}\big\}$ yields a set of generators for $Z$. Furthermore, 
\[P(\tilde u+I) = \tilde u + \varphi_1(H)X_{-\al}[X_\al,\tilde u]+\varphi_2(H)X_{-\al}^2[X_\al,[X_\al,\tilde u]]+\cdots+I.\]
Now we may use the bracket \eqref{eq:gp-p} along with the right half of the table in Figure \ref{fig:table1}.
\end{proof}

We revisit the discussion of Section \ref{sec:IrredZmod} through a concrete example of an irreducible $Z$-module related to oscilator representations.
For each non-negative integer $\la$, there exists a finite-dimensional irreducible representation $V(\la)$ of $\mathfrak{osp}(1|2)$. The dimension of $V(\la)$ equals $2\la+1$ and this list exhausts all finite-dimensional irreducible representations up to equivalence. The spectrum of $h$ on $V(\la)$ equals $\{\la,\la-1,\ldots,-\la\}$. 
The associative superalgebra $\C[x]$, where $x$ is declared odd, carries an action of of $\Fosp(1|2)$ determined by $x_\al\mapsto \frac{1}{\sqrt{2}}\frac{\partial}{\partial x}$ and $x_{-\al}\mapsto\frac{1}{\sqrt{2}}x$. The spectrum of $h$ on $\C[x]$ is $\frac{1}{2}+\Z_{\ge 0}$.
Therefore the spectrum of $h$ on $\tilde V(\la)=\C[x]\otimes V(\la)$ is a subset of $\frac{1}{2}+\Z$ and hence the action of $U\big(\Fosp(1|2)\times\Fosp(1|2)\big)$ on $\tilde V(\la)$ extends to the localization $U$ (at all $(h-n)$ for $n\in\Z$).
Furthermore, $\tilde V(\la)$ is locally finite with respect to the action of $x_\al$, hence is an object of the category $\CC$.
This means that the space of primitive vectors ($\Fg_+$-invariants) $\tilde V(\la)^+$ is an irreducible representation of the diagonal reduction algebra $Z$.

\begin{Example}\label{ex:rep}
	The superspace $V = \C[x] \otimes \C^{1|2}$ is an object of $\CC$. Define the elements $w_{1} = 1 \otimes v_{2}$ and $w_{2} = 1 \otimes v_{0} + \sqrt{2} x \otimes v_{2}$ and $w_{3}=-x^2\otimes v_2+\sqrt{2} x\otimes v_0-1\otimes v_1$. One can show that $V^{+} = \Span\{w_{1},w_{2},w_3\}$ (see \cite{fergusonWeightModulesOrthosymplectic2015} for the case of $\Fosp(1|2n)$ with $n>1$). The irreducible submodules corresponding to higher-dimensional analogues of $w_1$ and $w_2$ were studied in \cite{williamsBasesInfinitedimensionalRepresentations2020}. Note that $w_2$ is even while $w_1$ and $w_3$ are odd.
    In the ordered basis $(w_2,w_1,w_3)$, the irreducible representation $\rho\colon Z\to\End_\C(V^+)$ is given by
	\begin{align*}
		\rho(\bar x_\al)&=\begin{bmatrix}0&0&-6\\ 2&0&\phantom{-}0\\ 0&0&\phantom{-}0\end{bmatrix},&
		\rho(\bar x_{2\al})&=\begin{bmatrix}0&0&0\\ 0&0&2\\ 0&0&0\end{bmatrix},&
		\rho(\bar h)&=\begin{bmatrix} \frac{9}{2}&0&\phantom{-}0\\ 0&\frac{3}{2}&\phantom{-}0\\ 0&0&-\frac{9}{2}\end{bmatrix},&\\
		\rho(\bar x_{-\al})&=\begin{bmatrix}0&2&0\\ 0&0&0\\ 6&0&0\end{bmatrix},&
		\rho(\bar x_{-2\al})&=\begin{bmatrix}0&0&0\\ 0&0&0\\ 0&2&0\end{bmatrix},&
		\rho(f(H))&=\begin{bmatrix}f(\frac{1}{2})&0&0\\0&f(-\frac{1}{2})&0\\0&0&f(\frac{3}{2})\end{bmatrix}.
	\end{align*}
	
\end{Example}

\subsection{Ordered monomials and PBW bases}

\subsubsection{Ordered diamond products}

We compute the ordered products we need in order to deduce the commutation relations in the next section.
We order the generators of the diagonal reduction algebra as follows.
\begin{equation}
	\bar x_{-2\al} < \bar x_{-\al} < \bar h < \bar x_\al < \bar x_{2\al}.
\end{equation}
We compute the lexicographically ordered diamond product of any two generators.
Put $\tilde y = \delta_-(y)$ for $y\in\{x_{\pm\al},x_{\pm 2\al},h\}$.

\begin{Lemma} \label{lem:DiamondToU}
	\begin{align}
		\bar y \bDiamond \bar x_{2\al} &=
		\tilde y\tilde x_{2\al}+\textup{\Romanbar{II}},\qquad\forall y\in \{x_{\pm\al}, x_{\pm 2\al}, h\}, \label{eq:anyD2alpha}\\ 
		\bar x_{-2\al}\bDiamond \bar y &= \tilde x_{-2\al}\tilde y+\textup{\Romanbar{II}},\qquad\forall y\in \{x_{\pm\al}, x_{\pm 2\al}, h\}, \label{eq:2alphaDany} \\
		\bar x_\al \bDiamond \bar x_\al &=(\tilde x_\al)^2 -2 \varphi_1(H+1)\tilde h\tilde x_{2\al}+\textup{\Romanbar{II}},
		\label{eq:alphaDalpha} \\
		\bar h \bDiamond \bar x_\al &=\tilde h\tilde x_\al-2\varphi_1(H)\tilde x_{-\al}\tilde x_{2\al}+\textup{\Romanbar{II}},
		\label{eq:hDalpha} \\
		\bar x_{-\al} \bDiamond \bar x_\al &=\tilde x_{-\al}\tilde x_\al - 4\varphi_1(H-1)\tilde x_{-2\al}\tilde x_{2\al}+\textup{\Romanbar{II}},
		\label{eq:-alphaDalpha} \\
		\bar h\bDiamond \bar h &=\tilde h^2 + \varphi_1(H)\tilde x_{-\al}\tilde x_\al-4\varphi_2(H)\tilde x_{-2\al}\tilde x_{2\al}+\textup{\Romanbar{II}},
		\label{eq:hDh} \\
		\bar x_{-\al} \bDiamond \bar h &=\tilde x_{-\al}\tilde h + 
		2\varphi_1(H-1)\tilde x_{-2\al}\tilde x_\al
		+\textup{\Romanbar{II}},
		\label{eq:-alDh}
		\\
		\bar x_{-\al}\bDiamond\bar x_{-\al} &=(\tilde x_{-\al})^2 +2 \varphi_1(H-1)\tilde x_{-2\al}\tilde h+\textup{\Romanbar{II}}.
		\label{eq:-alD-al}
	\end{align}
	
\end{Lemma}

\begin{proof}
	\eqref{eq:anyD2alpha} and \eqref{eq:2alphaDany} follow directly from the formula \eqref{eq:diamond-compute} since $[X_\al, \tilde x_{2\al}]=0$ and $[\tilde x_{-2\alpha}, X_{-\al}]=0$.
	For \eqref{eq:alphaDalpha} we have
	\begin{align*}
		\bar x_{\al} \bDiamond \bar \al &=
		(\tilde x_\al)^2 + 
		[\tilde x_\al,X_{-\al}]\varphi_1(H+1)[X_\al,\tilde x_\al] +\Romanbar{II}\nonumber\\
		&=(\tilde x_\al)^2 -2 \tilde h\varphi_1(H+1)\tilde x_{2\al}+\Romanbar{II}.
	\end{align*}
	The proofs of \eqref{eq:hDalpha}, \eqref{eq:-alphaDalpha}, and \eqref{eq:hDh} are similar. Application of $\Theta$ to \eqref{eq:hDalpha} and \eqref{eq:alphaDalpha} yields \eqref{eq:-alDh} and \eqref{eq:-alD-al}, respectively.
Complete details of the proof are found in Appendix A. 
\end{proof}

It will be useful to invert these equations.

\begin{Lemma} \label{lem:tildetodiam}
	\begin{align}
		\tilde y\tilde x_{2\al}+\textup{\Romanbar{II}} &= \bar y \bDiamond \bar x_{2\al}, \qquad\forall y\in \{x_{\pm\al}, x_{\pm 2\al}, h\},\label{eq:RanyD2alpha} \\ 
		\tilde x_{-2\al}\tilde y+\textup{\Romanbar{II}}&=\bar x_{-2\al}\bDiamond \bar y,\qquad\forall y\in \{x_{\pm\al}, x_{\pm 2\al}, h\}, \label{eq:R2alphaDany} \\
		\tilde x_\al\tilde x_\al+\emph{\Romanbar{II}} &= \bar x_\al \bDiamond \bar x_\al +2\varphi_1(H+1)\bar h\bDiamond \bar x_{2\al},\label{eq:RalphaDalpha}\\
		\tilde h\tilde x_\al +\emph{\Romanbar{II}}&= \bar h\bDiamond\bar x_\al +2\varphi_1(H)\bar x_{-\al}\bDiamond\bar x_{2\al},\label{eq:RhDalpha}\\
		\tilde x_{-\al}\tilde x_\al+\emph{\Romanbar{II}} &=
		\bar x_{-\al} \bDiamond \bar x_\al +4\varphi_1(H-1)\bar x_{-2\al}\bDiamond\bar x_{2\al},\\
		\tilde h\tilde h +\emph{\Romanbar{II}}&=
		\bar h\bDiamond\bar h-\varphi_1(H)\bar x_{-\al}\bDiamond\bar x_{\al}+4\big(\varphi_2(H)-\varphi_1(H)\varphi_1(H-1)\big)\bar x_{-2\al}\bDiamond \bar x_{2\al},\\
		\tilde x_{-\al}\tilde h+\textup{\Romanbar{II}}&=\bar x_{-\al}\bDiamond\bar h -2\varphi_1(H-1)\bar x_{-2\al}\bDiamond\bar x_\al,\\
		\tilde x_{-\al}\tilde x_{-\al}+\textup{\Romanbar{II}}&=\bar x_{-\al}\bDiamond\bar x_{-\al}-2\varphi_1(H-1)\bar x_{-2\al}\bDiamond\bar h.
	\end{align}
\end{Lemma}

\begin{proof}
	By \eqref{eq:anyD2alpha} and \eqref{eq:alphaDalpha},
	\begin{align*}
		\tilde x_\al\tilde x_\al +\Romanbar{II}&=\bar x_\al\bDiamond\bar x_\al+(2\varphi_1(H+1)\tilde h\tilde x_{2\al}+\Romanbar{II}) \\
		&=\bar x_\al\bDiamond \bar x_\al+2\varphi_1(H+1)\bar h\bDiamond \bar x_{2\al}.
	\end{align*}
	Continued manipulation of the relations in Lemma \eqref{lem:DiamondToU} gives the rest of the inversions. Complete details of the proof are found in Appendix A.
\end{proof}

\subsection{Presentation}
We are now able to deduce the relations.
\begin{Theorem}\label{thm:relations}
	The following relations hold in $U/\emph{\Romanbar{II}}$:
	\begin{subequations}\label{eq:thm-rels}
		\begin{align}
			f(H)\bDiamond g(H)&=g(H)\bDiamond f(H),\qquad\forall f(H),g(H)\in D^{-1}U(\Fh),\\
			\bar x_{k\al}\bDiamond f(H) &= f(H+k)\bDiamond\bar x_{k\al}, \qquad \forall k\in\{\pm 1, \pm 2\},\;\forall f(H)\in D^{-1}U(\Fh),\\
			\bar h\bDiamond f(H) &= f(H)\bDiamond\bar h,\qquad\forall f(H)\in D^{-1}U(\Fh), \label{eq:barhH}\\
			\bar x_{2\al}\bDiamond \bar x_{\al}
			&=\big(1-\frac{2}{H+1}\big)\bar x_\al\bDiamond \bar x_{2\al}
			\label{eq:2alphaDalpha}\\
      \bar x_\al\bDiamond \bar x_\al &= \frac{2}{H}\bar h\bDiamond \bar x_{2\al} \label{eq:alphaDalpha-rel}\\
      \bar x_{-\al}\bDiamond \bar x_{-\al} &= -\frac{2}{H-2}\bar x_{-2\al}\bDiamond \bar h \label{eq:-alphaD-alpha-rel}\\
			\bar x_{2\al}\bDiamond \bar h &=
			\Big(1-\frac{2}{H+1}\Big)\bar h\bDiamond \bar x_{2\al}
			\label{eq:2alphaDh} \\
			\bar x_{2\al}\bDiamond \bar x_{-\al} &=
			\Big(1-\frac{2}{H(H-1)}\Big)
			\bar x_{-\al}\bDiamond \bar x_{2\al} 
			+\frac{2}{H+1} 
			\bar h\bDiamond \bar x_\al
			\label{eq:2alphaD-alpha}\\
			\bar x_{2\al}\bDiamond \bar x_{-2\al}&=
			\Big(1+2\frac{H^3 + H^2 - 6H + 4}{(H-2)(H-1)H(H+1)(H+2)}\Big)
			\bar x_{-2\al}\bDiamond \bar x_{2\al}
			\nonumber\\
			&\quad-\frac{H^2-H-1}{(H-1)H(H+1)}
\bar x_{-\al}\bDiamond \bar x_\al 
			+\frac{1}{H+1}\bar h\bDiamond \bar h 
			+\frac{-H^2}{H+1} 
			\label{eq:2alphaD-2alpha}\\
			\bar x_\al \bDiamond \bar h &=
			\big(1-\frac{1}{H}\big)\bar h\bDiamond \bar x_\al  
			\label{eq:alphaDh} \\
			\bar x_\al \bDiamond \bar x_{-\al} &=
			\Big(-1+\frac{-1}{H-1}\Big)
			\bar x_{-\al}\bDiamond \bar x_\al 
			+
   		\frac{4H}{(H-1)(H-2)}\bar x_{-2\al}\bDiamond \bar x_{2\al} 
			-\frac{1}{H}\bar h\bDiamond\bar h 
			+H
			\label{eq:alphaD-alpha} \\
			\bar x_\al \bDiamond \bar x_{-2\al} & = 
      \Big(1-\frac{2}{(H-1)(H-2)}\Big)\bar x_{-2\al}\bDiamond\bar x_\al 
			-\frac{2}{H}
			\bar x_{-\al}\bDiamond \bar h
			\label{eq:alphaD-2alpha}\\
			\bar h \bDiamond \bar x_{-\al} &=
			\big(1-\frac{1}{H-1}\big)\bar x_{-\al}\bDiamond \bar h 
			\label{eq:hD-alpha}\\
			\bar h \bDiamond \bar x_{-2\al} &=
			\Big(1-\frac{2}{H-1}\Big)\bar x_{-2\al}\bDiamond \bar h
			\label{eq:hD-2alpha}\\
			\bar x_{-\al} \bDiamond \bar x_{-2\al} &=\big(1-\frac{2}{H-2}\big)\bar x_{-2\al}\bDiamond \bar x_{-\al}
			\label{eq:-alphaD-2alpha}
		\end{align}
	\end{subequations}
\end{Theorem}

\begin{proof}
	Put $\tilde y = \delta_-(y)$ for $y\in\{x_{\pm\al},x_{\pm 2\al},h\}$.
	Putting $h=x_0$ for a moment, note that for any $\beta,\gamma\in\{\pm \al, \pm 2\al, 0\}$ such that $\beta+\gamma\neq 0$ we have
	\begin{equation}
		\tilde x_\beta \tilde x_\gamma + \Romanbar{II}=(-1)^{|\beta||\gamma|}\tilde x_\gamma\tilde x_\beta + \Romanbar{II}.
	\end{equation}
	This is due to the fact that, in $U$, we have $[\tilde x_\beta, \tilde x_\gamma] \in \C X_{\beta+\gamma}\subseteq \Romanbar{II}$ when $\beta+\gamma\neq 0$.
	%
	%
	Thus, we have
	\begin{align*}
		\bar x_{2\al}\bDiamond \bar x_{\al} &=
		\tilde x_{2\al}\tilde x_\al + [\tilde x_{2\al},X_{-\al}]\varphi_1(H+1)[X_\al,\tilde x_\al] +\Romanbar{II} \\
		&=\tilde x_\al \tilde x_{2\al}
		+2\varphi_1(H+2)\tilde x_\al\tilde x_{2\al}+\Romanbar{II}\\
		&=\big(1+2\varphi_1(H+2)\big)\bar x_\al\bDiamond \bar x_{2\al},
	\end{align*}
	where we used \eqref{eq:anyD2alpha}.

For a complete proof of the remaining relations, see Appendix A.
\end{proof}

\subsection{PBW basis for \texorpdfstring{$Z$}{Z}}

By the PBW theorem for $U(\mathfrak{G})$, it is immediate that $U/\Romanbar{II}$ is a free left $D^{-1}U(\Fh)$-module on
\begin{equation}
	\{\tilde x_{-2\al}^p\tilde x_{-\al}^q
	\tilde h^r\tilde x_{\al}^s\tilde x_{2\al}^t + \Romanbar{II}
	\mid p,q,r,s,t\in\Z_{\ge 0},\, q,s\le 1\}.
\end{equation}

\begin{Proposition} \label{prp:PBW}
	$U/\textup{\Romanbar{II}}$ is a free left $D^{-1}U(\Fh)$-module on the following set of monomials with respect to the diamond product:
	\begin{equation}
		\{\bar x_{-2\al}^{\bDiamond p}\bDiamond \bar x_{-\al}^{\bDiamond q}\bDiamond \bar h^{\bDiamond r}\bDiamond \bar x_\al^{\bDiamond s}\bDiamond \bar x_{2\al}^{\bDiamond t}
		\mid p,q,r,s,t\in\Z_{\ge 0},\, q,s\le 1\}.
	\end{equation}
\end{Proposition}

\begin{proof}
	By Theorem \ref{thm:relations}, the set spans $U/\Romanbar{II}$. By induction, each diamond monomial can be written 
	\begin{equation}
		\tilde x_{-2\al}^p\tilde x_{-\al}^q
		\tilde h^r\tilde x_\al^s\tilde x_{2\al}^t + \text{(lower terms)} +\Romanbar{II}
	\end{equation}
	where we order the monomials lexicographically with respect to $x_{-2\al}<x_{-\al}<h<x_\al<x_{2\al}$.
\end{proof}

\subsection{Main theorem}

\begin{Theorem}
	Let $D$ be the multiplicative subset of $U\big(\Fosp(1|2)\times\Fosp(1|2)\big)$ generated by $\{H-n\mid n\in\Z\}$ where $H=h\otimes 1+1\otimes h$ and let $Z$ be the diagonal reduction algebra $Z\big(\Fosp(1|2)\times\Fosp(1|2),\Fosp(1|2)\big)$.
	Then $Z$ is generated as a $\C$-algebra by $D^{-1}U(\Fh)\cup\{\bar x_{-2\al},\bar x_{-\al},\bar h,\bar x_\al,\bar x_{2\al}\}$ subject to relations \eqref{eq:thm-rels}.
\end{Theorem}

\begin{proof}
	Let $A$ be the $\C$-algebra generated by 
	$D^{-1}U(\Fh)\cup\{\bar x_{-2\al},\bar x_{-\al},\bar h,\bar x_\al,\bar x_{2\al}\}$ modulo relations \eqref{eq:thm-rels}. By Theorem \ref{thm:relations}, there exists a surjective $\C$-algebra homomorphism $\varphi:A\to Z$. Suppose $a\in A$ belongs to the kernel of this map. Using the relations \eqref{eq:thm-rels} we map write $a$ as a linear combinations of ordered monomials with coefficients on the left from $D^{-1}U(\Fh)$. Applying $\varphi$ we get a corresponding linear combination of ordered monomials in $Z$. But by Proposition \ref{prp:PBW} these are linearly independent over $D^{-1}U(\Fh)$, so all the coefficients are zero. Thus $a=0$. This proves that $\varphi$ is an isomorphism.
\end{proof}


\section{Applications}

\subsection{Dilation automorphisms of the reduction algebra}

\begin{Proposition}
Any automorphism $\tau$ of the diagonal reduction algebra of $\Fosp(1|2)$ of the form
$\tau(\bar x_\be)=k_\be \bar x_\be$, $\tau(\bar h)=k_h \bar h$, and $\tau(H)=H$ is given by
$k_h=\epsilon\in\{\pm 1\}$, $k_\al=\xi$, for some nonzero $\xi\in\C$, $k_{-\al}=\xi^{-1}$, $k_{2\al}=\epsilon \xi^2$, and $k_{-2\al}=\epsilon \xi^{-2}$.
Conversely, for any $\epsilon\in\{\pm 1\}$ and any nonzero $\xi\in\C$, there is an automorphism $\tau$ given by
$\tau(\bar x_{\pm \al})=\xi^{\pm 1} \bar x_{\pm \al}$, $\tau(\bar h)=\epsilon \bar h$, $\tau(\bar x_{\pm 2\al}) = \epsilon \xi^{\pm 2}\bar x_{\pm 2\al}$, and $\tau(H)=H$.
The group of these dilation automorphisms is isomorphic to $(\Z/2\Z)\times \C^{\ast}$.
\end{Proposition}

\begin{proof}
Using \eqref{eq:alphaD-alpha} to write $\tau(\bar x_\al\bDiamond \bar x_{-\al}) - k_\al k_{-\al}\bar x_\al\bDiamond \bar x_{-\al}$ in the PBW basis, the PBW Theorem (Proposition \ref{prp:PBW}) implies that $k_\al k_{-\al} = k_{2\al}k_{-2\al} = k_h^2 = 1$. So $k_h=\epsilon$ for some $\epsilon\in\{\pm 1\}$.
By \eqref{eq:alphaDalpha-rel}, we get $k_\al^2 = k_h k_{2\al}$. Since $k_h^2=1$, the previous equation\ implies $k_{2\al}=k_h k_\al^2$.
%
Conversely, we encourage the reader to verify that any such $\tau$ preserves the relations in Theorem \ref{thm:relations} as we omit the calculations here.
These automorphisms commute (like complex numbers under multiplication) and are uniquely determined by a choice of $\epsilon \in \{\pm 1\}$ and a complex unit $\xi \in \C^{\ast}$, proving the final statement. 
\end{proof}

\begin{Example} The square $\Theta^2$ of the anti-automorphism $\Theta$ (Section \ref{sec:anti-automorphism}) is an automorphism of the form above corresponding to $\epsilon=1$ and $\xi=-1$.
\end{Example}

\subsection{Linear Casimir}

Put
\begin{equation}
C^{(1)}= \hat h = (H-1)\bar h.
\end{equation}
Use of relations \eqref{eq:barhH} and \eqref{eq:alphaDh} show $C^{(1)}$ commutes with $\bar h$, $H$, and $\bar x_\al$. Similarly, $C^{(1)}$ commutes with $\bar x_{2\al}$.
By applying the anti-automorphism $\Theta$, we see $C^{(1)}$ is a central element as it also commutes with $\bar x_{-\al}$ $\bar x_{-2\al}$. 

\begin{Example}
Continuing Example \ref{ex:rep}, we can check
$\rho(C^{(1)})=\begin{bmatrix}-\frac{9}{4} & \phantom{-}0 & \phantom{-}0\\\phantom{-}0 &-\frac{9}{4} & \phantom{-}0\\ \phantom{-}0 & \phantom{-}0 & -\frac{9}{4}
\end{bmatrix}$.
Precisely, $\rho(C^{(1)})= \rho\left((H-1)\bar{h}\right) = \begin{bmatrix}-\frac{1}{2}&\phantom{-}0&0\\\phantom{-}0&-\frac{3}{2}&0\\\phantom{-}0&\phantom{-}0&\frac{1}{2}\end{bmatrix}\begin{bmatrix}\frac{9}{2}&0&0\\0&\frac{3}{2}&0\\0&0&-\frac{9}{2}\end{bmatrix} =
\begin{bmatrix}-\frac{9}{4} & \phantom{-}0 & \phantom{-}0\\\phantom{-}0 &-\frac{9}{4} & \phantom{-}0\\ \phantom{-}0 & \phantom{-}0 & -\frac{9}{4} 
\end{bmatrix}$.
\end{Example}

\subsection{Quadratic anti-Casimir}
Recall that an even element is called \emph{anti-central} if it commutes with even elements and anti-commutes with odd elements.

\begin{Lemma}\label{lem:QuadAnsatz}
Consider the following ansatz of a quadratic anti-central element in $U/\emph{\Romanbar{II}}$: 
\begin{equation}
C^{(2)} = f_2(H) \bar x_{-2\al} \bDiamond \bar x_{2\al} + f_1(H) \bar x_{-\al} \bDiamond \bar x_{\al} + f_0(H)\hat h \bDiamond \hat h+g(H).
\end{equation}
Then $C^{(2)}$ is anti-central if and only if the following equations hold:
 \begin{align}
f_1(H-1)&=\frac{H-1}{H-2}f_1(H)+\frac{H-2}{H-1}f_2(H) \label{eq:f1}  \\
f_2(H-1)&=-\frac{4(H-1)}{(H-2)(H-3)}f_1(H)-\frac{H-1}{H-3}f_2(H) \label{eq:f2} \\
f_0(H-1)&=-f_0(H)+\frac{1}{(H-1)(H-2)^2}f_1(H) \label{eq:f0} \\
g(H-1)&=-g(H)-f_1(H)(H-1) \label{eq:g}
\end{align}

\end{Lemma}

\begin{proof}
    The details of the proof are found in Appendix B.
\end{proof}

Next we will show existence and uniqueness (up to complex scalar multiple) of the above system of equations. 

Consider the following element of $U(\Fosp(1|2))$ from \cite{lesniewskiRemarkCasimirElements1995}: \footnote{The following converts the basis of $\Fosp(1|2)$ of Le\'{s}niewski to ours:  $L_3=\frac{1}{2}h,\, G_\pm = -\frac{\sqrt{-1}}{2} x_{\mp\al},\, L_\pm = -x_{\mp 2\al}$.}

\begin{align}
L &=-\frac{1}{2}(x_\al x_{-\al}-x_{-\al}x_\al -\frac{1}{2})\nonumber \\
&=x_{-\al}x_\al - \frac{1}{2}h+\frac{1}{4}
\end{align}
One can check that 
\begin{equation}
L x_\al + x_\al L= 0.
\end{equation}

Like so, using \eqref{comm-xalxnal}, \eqref{comm-hx}: 
\begin{align*}
L x_{\al} + x_{\al} L &=  x_{-\al}x_{\al}x_{\al} - \frac{1}{2}hx_{\al}+\frac{1}{4}x_{\al} + x_{\al}x_{-\al}x_{\al} - \frac{1}{2}x_{\al}h+\frac{1}{4}x_{\al}\\
& =  x_{-\al}x_{\al}x_{\al} - x_{-\al}x_{\al}x_{\al} + hx_{\al} - hx_{\al} - \frac{1}{2}x_{\al} + \frac{1}{2}x_{\al}\\
& = 0.
\end{align*}

Since $L$ is fixed by the anti-automorphism $\theta$, $L$ also anti-commutes with $x_{-\al}$, hence the commuting of $L$ with $x_{\pm 2\al}$ and $h$.
Recall $U(\Fosp(1|2)) = U(\Fosp(1|2))_{\overline{0}} \oplus U(\Fosp(1|2))_{\overline{1}}$ is a superalgebra, so the statements above imply $L$ belongs to the center of the subalgebra $U(\Fosp(1|2))_{\overline{0}}$ and anti-commutes with elements of $U(\Fosp(1|2))_{\overline{1}}$. Thus $L$ is an anti-central element of $U(\Fosp(1|2))$.

Consider now
\begin{equation}
\CL=L\otimes L \in U.
\end{equation}
Then $\CL$ anti-commutes with $X_{\al}$, hence commutes with $X_{2\al}$, hence $\CL$ belongs to the normalizer $N$ of $I=U\Fg_+$ in $U$.
Thus 
\begin{equation}
\mathsf{P}_M(\CL+\Romanbar{II})=\CL+I,
\end{equation}
or equivalently,
\begin{equation}
P_M(\CL+I)=\CL+I.
\end{equation}

We will now simplify $\bar\CL=\CL+\Romanbar{II}$ and write it in the PBW basis for $U/\Romanbar{II}$.
Below let $\equiv$ stand for congruence modulo $\Romanbar{II}$.
\begin{align*}
\CL &\equiv (L\otimes 1)(1\otimes L)\\
&\equiv (x_{-\al}x_\al -\frac{1}{2}h+\frac{1}{4})\otimes 1 \cdot 1\otimes (x_{-\al}x_\al-\frac{1}{2}h+\frac{1}{4}) \\
&\equiv \big(\frac{1}{4}(X_{-\al}+\tilde x_{-\al})(X_\al+\tilde x_\al) -\frac{1}{4}(H+\tilde h)+\frac{1}{4}\big)
\big( \frac{1}{4}(X_{-\al}-\tilde x_{-\al})(X_\al-\tilde x_\al)-\frac{1}{4}(H-\tilde h)+\frac{1}{4}\big)\\
&\equiv
 \big(\frac{1}{4}\tilde x_{-\al}(X_\al+\tilde x_\al) -\frac{1}{4}(H+\tilde h)+\frac{1}{4}\big)
\big( \frac{1}{4}(-X_{-\al}+\tilde x_{-\al})\tilde x_\al-\frac{1}{4}(H-\tilde h)+\frac{1}{4}\big).
\end{align*}
Thus
\begin{align*}
16\CL &\equiv \tilde x_{-\al}(X_\al+\tilde x_\al)(-X_{-\al}+\tilde x_{-\al})\tilde x_\al \\
&\quad + \tilde x_{-\al}(X_\al+\tilde x_\al) (1-H+\tilde h)\\
&\quad + (1-H-\tilde h)(-X_{-\al}+\tilde x_{-\al})\tilde x_\al\\
&\quad + (1-H-\tilde h)(1-H+\tilde h)\\
&\equiv
\tilde x_{-\al}(X_\al+\tilde x_\al)(-X_{-\al}+\tilde x_{-\al})\tilde x_\al \\
&\quad + \tilde x_{-\al}\tilde x_\al (1-H+\tilde h)+\tilde x_{-\al} [X_\al,\tilde h]\\
&\quad + (1-H-\tilde h)\tilde x_{-\al}\tilde x_\al+[\tilde h,X_{-\al}]\tilde x_\al \\
&\quad + (1-H-\tilde h)(1-H+\tilde h)\\
&\equiv
\tilde x_{-\al}(X_\al+\tilde x_\al)(-X_{-\al}+\tilde x_{-\al})\tilde x_\al \\
&\quad + \tilde x_{-\al}\tilde x_\al (-H+\tilde h)+4\tilde x_{-\al} \tilde x_\al\\
&\quad + (-H-\tilde h)\tilde x_{-\al}\tilde x_\al\\
&\quad + (1-H-\tilde h)(1-H+\tilde h).
\end{align*}
Using that
$X_\al+\tilde x_\al=2 x_\al\otimes 1$ and $X_{-\al}-\tilde x_{-\al} = 2\cdot 1\otimes x_{-\al}$ anti-commute, we get
\begin{align*}
16\CL &\equiv
\tilde x_{-\al}(-\tilde x_{-\al}X_\al+X_{-\al}X_\al-\tilde x_{-\al}\tilde x_\al + X_{-\al}\tilde x_\al)\tilde x_\al \\
&\quad + [\tilde x_{-\al}\tilde x_\al,\tilde h] -2(H-2)\tilde x_{-\al}\tilde x_\al \\
&\quad -\tilde h\tilde h+(H-1)^2.
\end{align*}
Since $\tilde x_{-\al}\tilde x_{-\al}=\frac{1}{2}[\tilde x_{-\al},\tilde x_{-\al}]=X_{-2\al}\in \Fg_-$ and similarly $\tilde x_\al \tilde x_\al \in \Fg_+$, we get
\begin{align*}
16\CL &\equiv
\tilde x_{-\al}X_{-\al}X_\al\tilde x_\al \\
&\quad + [\tilde x_{-\al},\tilde h]\tilde x_\al+\tilde x_{-\al}[\tilde x_\al,\tilde h] -2(H-2)\tilde x_{-\al}\tilde x_\al  -\tilde h\tilde h+(H-1)^2\\
&\equiv (-X_{-\al}\tilde x_{-\al} + 2\tilde x_{-2\al})(-\tilde x_\al X_\al -2\tilde x_{2\al})-2(H-2)\tilde x_{-\al}\tilde x_\al -\tilde h\tilde h +(H-1)^2\\
&\equiv -4\tilde x_{-2\al}\tilde x_{2\al}-2(H-2)\tilde x_{-\al}\tilde x_\al -\tilde h\tilde h +(H-1)^2.
\end{align*}
%
The following theorem gives the explicit expression of $\bar\CL$ in the PBW basis.
\begin{Theorem}
The unique (up to complex scalar multiple) anti-central element $C^{(2)}$ of the form in Lemma \ref{lem:QuadAnsatz} is:
\begin{equation}\label{eq:16L}
    16\bar\CL = 4\frac{H-2}{H-1} \bar x_{-2\al} \bDiamond \bar x_{2\al} - \left(2(H-2) + \frac{1}{H-1}\right) \bar x_{-\al} \bDiamond \bar x_{\al} - \frac{1}{(H-1)^{2}}\hat{h}\bDiamond\hat{h} + (H-1)^{2}.
\end{equation}
\end{Theorem}

\begin{proof}
 As already observed, $\bar\CL$ is an anti-central element. Using Lemma \ref{lem:tildetodiam} one can write $\bar\CL$ in the PBW basis from Proposition \ref{prp:PBW}, obtaining the identity \eqref{eq:16L}.  Therefore, by Lemma \ref{lem:QuadAnsatz}, the coefficients of $\bar\CL$ in the right hand side of \eqref{eq:16L} satisfy equations \eqref{eq:f1}-\eqref{eq:g}.
  The uniqueness follows from the fact that, given initial values for $f_i(\frac{1}{2})$, $i=0,1,2$, and $g(\frac{1}{2})$, the system \eqref{eq:f1}-\eqref{eq:g} uniquely determines $f_i(a)$ and $g(a)$ for every half-integer $a\in\frac{1}{2}+\Z$. Indeed, the $2\times2$ coefficient matrix has determinant $\frac{-(H-3)}{\;H-2}$. The rational functions $f_i$ and $g$ are then uniquely determined by those values $f_i(a)$ and $g(a)$. 
\end{proof}

\begin{Example}
Continuing Example \ref{ex:rep}, one can check that $\rho(\bar\CL)=\begin{bmatrix}0&0&0\\0&0&0\\0&0&0\end{bmatrix}$: 
\begin{align*}
    \rho(16\bar\CL) &= \rho\left(4\frac{H-2}{H-1} \bar x_{-2\al} \bDiamond \bar x_{2\al} - \left(2(H-2) + \frac{1}{H-1}\right) \bar x_{-\al} \bDiamond \bar x_{\al} - \frac{1}{(H-1)^{2}}\hat{h}\bDiamond\hat{h} + (H-1)^{2}\right)\\
    & = \begin{bmatrix}12&0&\phantom{-}0\\0&\frac{20}{3}&\phantom{-}0\\0&0&-4\end{bmatrix} 
    \begin{bmatrix}0&0&0\\ 0&0&0\\ 0&2&0\end{bmatrix} \begin{bmatrix}0&0&0\\ 0&0&2\\ 0&0&0\end{bmatrix} 
    -\begin{bmatrix}-5&\phantom{-}0&0\\\phantom{-}0&-\frac{17}{3}&0\\\phantom{-}0&\phantom{-}0&1\end{bmatrix} 
    \begin{bmatrix}0&2&0\\ 0&0&0\\ 6&0&0\end{bmatrix} \begin{bmatrix}0&0&-6\\ 2&0&\phantom{-}0\\ 0&0&\phantom{-}0\end{bmatrix}\\
    & \quad -  \begin{bmatrix}4&0&0\\0&\frac{4}{9}&0\\0&0&4\end{bmatrix} 
    \begin{bmatrix}-\frac{9}{4} & \phantom{-}0 & \phantom{-}0\\\phantom{-}0 &-\frac{9}{4} & \phantom{-}0\\ \phantom{-}0 & \phantom{-}0 & -\frac{9}{4}\end{bmatrix} 
    \begin{bmatrix}-\frac{9}{4} & \phantom{-}0 & \phantom{-}0\\\phantom{-}0 &-\frac{9}{4} & \phantom{-}0\\ \phantom{-}0 & \phantom{-}0 & -\frac{9}{4}\end{bmatrix} + \begin{bmatrix}\frac{1}{4}&0&0\\0&\frac{9}{4}&0\\0&0&\frac{1}{4}\end{bmatrix}\\
    & = \begin{bmatrix}0&0&\phantom{-}0\\0&0&\phantom{-}0\\0&0&-16\end{bmatrix}  - \begin{bmatrix}-20&0&0\\\phantom{-}0&0&0\\\phantom{-}0&0&-36\end{bmatrix} - \begin{bmatrix}\frac{81}{4}&0&0\\0&\frac{9}{4}&0\\0&0&\frac{81}{4}\end{bmatrix} + \begin{bmatrix}\frac{1}{4}&0&0\\0&\frac{9}{4}&0\\0&0&\frac{1}{4}\end{bmatrix}\\
    & = \begin{bmatrix}0&0&0\\0&0&0\\0&0&0\end{bmatrix}.
\end{align*}

\end{Example}

\section*{Appendix A}

\begin{proof}[Details for the proof of Lemma \ref{lem:DiamondToU}]
	For \eqref{eq:hDalpha} we have
	\begin{align*}
		\bar h \bDiamond \bar x_\al &=
		\tilde h\tilde x_\al + [\tilde h,X_{-\al}]\varphi_1(H+1)[X_\al,\tilde x_\al]+\Romanbar{II} \\
		&=\tilde h\tilde x_\al-2\tilde x_{-\al}\varphi_1(H+1)\tilde x_{2\al}+\Romanbar{II} \\
		&=\tilde h\tilde x_\al-2\varphi_1(H)\tilde x_{-\al}\tilde x_{2\al}+\Romanbar{II}.
	\end{align*}
	For \eqref{eq:-alphaDalpha} we have
	\begin{align*}
		\bar x_{-\al} \bDiamond \bar x_\al &=
		\tilde x_{-\al}\tilde x_\al + [\tilde x_{-\al},X_{-\al}]\varphi_1(H+1)[X_\al,\tilde x_\al] +\Romanbar{II} \\
		&=\tilde x_{-\al}\tilde x_\al - 4\tilde x_{-2\al}\varphi_1(H+1)\tilde x_{2\al}+\Romanbar{II} \\
		&=\tilde x_{-\al}\tilde x_\al - 4\varphi_1(H-1)\tilde x_{-2\al}\tilde x_{2\al}+\Romanbar{II}.
	\end{align*}
	To prove \eqref{eq:hDh} we use three terms from the expansion \eqref{eq:diamond-compute}:
	\begin{align*}
		\bar h\bDiamond \bar h &= (\tilde h)^2 + [\tilde h,X_{-\al}]\varphi_1(H+1)[X_\al,\tilde h]+\\
		&\qquad\qquad +[[\tilde h,X_{-\al}],X_{-\al}]\varphi_2(H+2)[X_\al,[X_\al,\tilde h]] +\Romanbar{II}\\
		&=(\tilde h)^2 + \tilde x_{-\al}\varphi_1(H+1)\tilde x_\al-4\tilde x_{-2\al}\varphi_2(H+2)\tilde x_{2\al}+\Romanbar{II}\\
		&=(\tilde h)^2 + \varphi_1(H)\tilde x_{-\al}\tilde x_\al-4\varphi_2(H)\tilde x_{-2\al}\tilde x_{2\al}+\Romanbar{II}.
	\end{align*}
	Using \eqref{eq:hDalpha} we have, recalling that $\theta(x_{2\al})=-x_{-2\al}$ and $\theta(x_\al)=\sqrt{-1}x_{-\al}$,
	\begin{align*}
		\bar x_{-\al} \bDiamond \bar h &= -\sqrt{-1}\Theta(\bar h\bDiamond \bar x_\al) \\
		&=\tilde x_{-\al}\tilde h + 2\tilde x_{-2\al}\varphi_1(H+1)\tilde x_\al+\Romanbar{II} \\
		&=\tilde x_{-\al}\tilde h + 2\varphi_1(H-1)\tilde x_{-2\al}\tilde x_\al+\Romanbar{II}.
	\end{align*}
	Similarly, using \eqref{eq:alphaDalpha} we have
	\begin{align*}
		\bar x_{-\al}\bDiamond\bar x_{-\al} &= \Theta(\bar x_\al \bDiamond \bar x_\al) \\
		&=(\tilde x_{-\al})^2 +2 \tilde x_{-2\al}\varphi_1(H+1)\tilde h+\Romanbar{II} \\
		&=(\tilde x_{-\al})^2 +2 \varphi_1(H-1)\tilde x_{-2\al}\tilde h+\Romanbar{II}.
	\end{align*}
\end{proof}

\begin{proof}[Details of the proof for Lemma \eqref{lem:tildetodiam}]
	By \eqref{eq:hDalpha},
	\begin{align*}
		\tilde h\tilde x_\al+\Romanbar{II}&=\bar h\bDiamond\bar x_\al+(2\varphi_1(H)\tilde x_{-\al}\tilde x_{2\al}+\Romanbar{II})\\
		&=\bar h\bDiamond\bar x_\al+2\varphi(H)\bar x_{-\al}\bDiamond\bar x_{2\al}.
	\end{align*}
	By \eqref{eq:-alphaDalpha},
	\begin{align*}
		\tilde x_{-\al}\tilde x_\al+\Romanbar{II}&=\bar x_{-\al}\bDiamond \bar x_\al +\big(4\varphi_1(H-1)\tilde x_{-2\al}\tilde x_{2\al}+\Romanbar{II}\big)\\
		&=\bar x_{-\al}\bDiamond\bar x_{\al}+4\varphi(H-1)\bar x_{-2\al}\bDiamond x_{2\al}.
	\end{align*}
	By \eqref{eq:hDh} and \eqref{eq:-alphaDalpha},
	\begin{align*}
		\tilde h\tilde h+\Romanbar{II}&=\bar h\bDiamond\bar h-\big(\varphi_1(H)\tilde x_{-\al}\tilde x_\al-4\varphi_2(H)\tilde x_{-2\al}\tilde x_{2\al}+\Romanbar{II}\big) \\
		&=\bar h\bDiamond\bar h -\varphi_1(H)\big(\bar x_{-\al}\bDiamond \bar x_\al + 4\varphi_1(H-1)\bar x_{-2\al}\bDiamond\bar x_{2\al}\big) + 4\varphi_2(H)\bar x_{-2\al}\bDiamond\bar x_{2\al}\\
		&=\bar h\bDiamond\bar h-\varphi_1(H)\bar x_{-\al}\bDiamond\bar x_\al +4\big(\varphi_2(H)-\varphi_1(H)\varphi_1(H-1)\big)\bar x_{-2\al}\bDiamond\bar x_{2\al}.
	\end{align*}
	By \eqref{eq:RhDalpha},
	\begin{align*}
		\tilde x_{-\al}\tilde h+\Romanbar{II}&=\Theta(\tilde h\tilde x_\al+\Romanbar{II})\\
		&=\Theta\big(\bar h\bDiamond\bar x_\al+2\varphi_1(H)\bar x_{-\al}\bDiamond\bar x_{2\al}\big)\\
		&=\bar x_{-\al}\bDiamond\bar h-\bar x_{-2\al}\bDiamond\bar x_\al\cdot 2\varphi_1(H)\\
		&=\bar x_{-\al}\bDiamond\bar h-2\varphi_1(H-1)\bar x_{-2\al}\bDiamond\bar x_\al.
	\end{align*}
	By \eqref{eq:RalphaDalpha},
	\begin{align*}
		\tilde x_{-\al}\tilde x_{-\al}+\Romanbar{II}&=\Theta(\tilde x_\al\tilde x_\al+\Romanbar{II})\\
		&=\Theta\big(\bar x_\al\bDiamond\bar x_\al+2\varphi_1(H+1)\bar h\bDiamond\bar x_{2\al}\big)\\
		&=\bar x_{-\al}\bDiamond\bar x_{-\al}-\bar x_{-2\al}\bDiamond\bar h\cdot 2\varphi_1(H+1)\\
		&=\bar x_{-\al}\bDiamond\bar x_{-\al}-2\varphi_1(H-1)\bar x_{-2\al}\bDiamond\bar h.
	\end{align*}
\end{proof}

\begin{proof}[Details of the proof of Theorem \ref{thm:relations}]

	We have
	\begin{align*}
		\bar x_\al\bDiamond\bar x_\al &= 
		\tilde x_\al P\tilde x_\al +\Romanbar{II}\\
		&=\tilde x_\al^2+[\tilde x_\al,X_{-\al}]\varphi_1(H+1)[X_\al,\tilde x_\al]+\Romanbar{II}\\
		&=\tilde x_\al^2-2\tilde h\varphi_1(H+1)\tilde x_{2\al}+\Romanbar{II}\\
		&=-2\varphi_1(H+1)\bar h\bDiamond \bar x_{2\al}
	\end{align*}
	where we used that $\tilde x_\al^2=\frac{1}{2}[\tilde x_\al,\tilde x_\al]=-X_{2\al}\in\Fg_+\subset\Romanbar{II}$, and \eqref{eq:anyD2alpha}. Applying $\Theta$ to \eqref{eq:alphaDalpha-rel} gives \eqref{eq:-alphaD-alpha-rel}.
	
	Next,
	\begin{align*}
		\bar x_{2\al}\bDiamond \bar h &=
		\tilde x_{2\al}\tilde h +
		[\tilde x_{2\al},X_{-\al}]\varphi_1(H+1)[X_\al,\tilde h]+
		[[\tilde x_{2\al},X_{-\al}],X_{-\al}]\varphi_2(H+2)[X_\al,[X_\al,\tilde h]] +\Romanbar{II}\nonumber\\
		&=\tilde x_{2\al}\tilde h +(-\tilde x_\al) \varphi_1(H+1) \tilde x_\al +(-\tilde h) \varphi_2(H+2) (-2\tilde x_{2\al}) +\Romanbar{II}\nonumber\\
		&=\tilde h\tilde x_{2\al} -\varphi_1(H+2)\tilde x_\al \tilde x_\al +2\varphi_2(H+2)\tilde h\tilde x_{2\al} +\Romanbar{II}\nonumber\\
		&=	\Big(1+2\varphi_2(H+2)\Big)\bar h\bDiamond\bar x_{2\al}
	\end{align*}
	where we used \eqref{eq:anyD2alpha} and that $\tilde x_\al^2\in\Romanbar{II}$ as before.
	
	Next,
	\begin{align*}
		\bar x_{2\al}\bDiamond \bar x_{-\al} &=
		\tilde x_{2\al}\tilde x_{-\al} \\
		&\quad+ 
		[\tilde x_{2\al},X_{-\al}]\varphi_1(H+1)[X_\al,\tilde x_{-\al}]  \\
		&\quad +
		[[\tilde x_{2\al},X_{-\al}],X_{-\al}]\varphi_2(H+2)[X_\al,[X_\al,\tilde x_{-\al}]] \\
		&\quad + 
		[\,\cdot\,, X_{-\al}]^3(\tilde x_{2\al})\cdot \varphi_3(H+3)\cdot [X_\al,\,\cdot\,]^3(\tilde x_{-\al}) +\Romanbar{II}\\
		&=\tilde x_{2\al}\tilde x_{-\al} \\
		&\quad+(-\tilde x_\al)\varphi_1(H+1)\tilde h \\
		&\quad+(-\tilde h)\varphi_2(H+2) \tilde x_\al \\
		&\quad+(-\tilde x_{-\al}) \varphi_3(H+3) (-2\tilde x_{2\al}) +\Romanbar{II}\\
		&= \Big(1+2\varphi_3(H+2)\Big)\tilde x_{-\al}\tilde x_{2\al} +\Big(-\varphi_1(H+2)-\varphi_2(H+2)\Big)\tilde h\tilde x_{\al} +\Romanbar{II} \\
		&=\Big(1+2\varphi_3(H+2)\Big)\bar x_{-\al}\bDiamond\bar x_{2\al}\\
		&\quad
		+\Big(-\varphi_1(H+2)-\varphi_2(H+2)\Big)
		\Big(\bar h\bDiamond\bar x_\al + 2\varphi_1(H)\bar x_{-\al}\bDiamond \bar x_{2\al}\Big)\\
		&=\Big(1+2\varphi_3(H+2)-2\varphi_1(H+2)\varphi_1(H)-2\varphi_2(H+2)\varphi_1(H)\Big)\bar x_{-\al}\bDiamond\bar x_{2\al} \\
		&\quad+\Big(-\varphi_1(H+2)-\varphi_2(H+2)\Big)\bar h\bDiamond \bar x_{2\al}
	\end{align*}
	using \eqref{eq:anyD2alpha} and \eqref{eq:hDalpha}.
	
	
	Next,
	\begin{align*}
		\bar x_{2\al} \bDiamond \bar x_{-2\al} &=
		\sum_{n=0}^4
		[\,\cdot\,,X_{-\al}]^n(\tilde x_{2\al})\cdot \varphi_n(H+n) \cdot [X_\al,\,\cdot\,]^n(x_{-2\al})+\Romanbar{II}\\
		&=\tilde x_{2\al}\tilde x_{-2\al}
		-\tilde x_\al\varphi_1(H+1)\tilde x_{-\al}
		-\tilde h\varphi_2(H+2)\tilde h
		-\tilde x_{-\al}\varphi_3(H+3)\tilde x_\al\\
		&\quad+4\tilde x_{-2\al}\varphi_4(H+4)\tilde x_{2\al}+\Romanbar{II}
		\\
		&=\tilde x_{-2\al}\tilde x_{2\al} - H - \varphi_1(H+2) \big(-\tilde x_{-\al}\tilde x_{\al} +H\big) - \varphi_2(H+2)\tilde h^2 -\varphi_3(H+2) \tilde x_{-\al}\tilde x_\al \\ 
		&\quad+ 4\varphi_4(H+2) \tilde x_{-2\al}\tilde x_{2\al}+\Romanbar{II}\\
		&=-H\big(1+\varphi_1(H+2)\big) +\big(1+4\varphi_4(H+2)\big)\tilde x_{-2\al}\tilde x_{2\al} \\
		&\quad + \big(\varphi_1(H+2)-\varphi_3(H+2)\big)\tilde x_{-\al}\tilde x_\al
		-\varphi_2(H+2)\tilde h^2 + \Romanbar{II}\\
		&=-H\big(1+\varphi_1(H+2)\big) 
		+\big(1+4\varphi_4(H+2)\big)\bar x_{-2\al}\bDiamond \bar x_{2\al} \\
		&\quad+ \big(\varphi_1(H+2)-\varphi_3(H+2)\big)\big(\bar x_{-\al}\bDiamond \bar x_\al+4\varphi_1(H-1)\bar x_{2\al}\bDiamond\bar x_{2\al}\big) \\
		&\quad -\varphi_2(H+2)\Big(\bar h\bDiamond\bar h-\varphi_1(H)\bar x_{-\al}\bDiamond \bar x_\al+\big(4\varphi_2(H)-4\varphi_1(H)\varphi_1(H-1)\big)\bar x_{-2\al}\bDiamond\bar x_{2\al}\Big)\\
		&=-H\big(1+\varphi_1(H+2)\big) +
		\Big(1+4\varphi_4(H+2)+4\varphi_1(H+2)\varphi_1(H-1) \\
		&\quad-4\varphi_3(H+3)\varphi_1(H-1)-4\varphi_2(H+2)\varphi_2(H)+4\varphi_2(H+2)\varphi_1(H)\varphi_1(H-1)
		\Big)\bar x_{-2\al}\bDiamond\bar x_{2\al} \\
		&\quad+\Big(\varphi_1(H+2)-\varphi_3(H+2)+\varphi_2(H+2)\varphi_1(H)\Big)\bar x_{-\al}\bDiamond\bar x_\al -\varphi_2(H+2)\bar h\bDiamond\bar h
	\end{align*} 
	where we used \eqref{eq:anyD2alpha}, \eqref{eq:hDh}, and \eqref{eq:-alphaDalpha} in the penultimate equality to express products in the tilde variables in terms of the diamond products.
	
	Next,
	\begin{align*}
		\bar x_\alpha \bDiamond \bar h &= 
		\tilde x_\al\tilde h+[\tilde x_\al,X_{-\al}]\varphi_1(H+1)[X_\al,\tilde h]+[[\tilde x_{\al},X_{-\al}],X_{-\al}]\varphi_2(H+2)[X_\al,[X_\al,\tilde h]]+\Romanbar{II} \\
		&=\big(1+ \varphi_1(H+1)\big)\tilde h \tilde x_\al -2 \varphi_2(H+1)\tilde x_{-\al} \tilde x_{2\al}\\
		&=\big(1+\varphi_1(H+1)\big)\big(\bar h\bDiamond \bar x_{\al}+2\varphi_1(H)\bar x_{-\al}\bDiamond \bar x_{2\al}\big) 
		-2\varphi_2(H+1) \bar x_{-\al}\bDiamond \bar x_{2\al} \\
		&=\big(1+\varphi_1(H+1)\big)\bar h\bDiamond\bar x_\al +\big(2\varphi_1(H)+2\varphi_1(H+1)\varphi_1(H)-2\varphi_1(H+1)\big)\bar x_{-\al}\bDiamond\bar x_{2\al}
	\end{align*}
	where we used \eqref{eq:hDalpha} and \eqref{eq:anyD2alpha}.
	Substituting $\varphi_1(H)=\varphi_2(H)=\frac{-1}{H-1}$, one checks that $2\varphi_1(H)+2\varphi_1(H+1)\varphi_1(H)-2\varphi_2(H+1)=0$, proving \eqref{eq:alphaDh}.
	
	Next,
	\begin{align*}
		\bar x_\al \bDiamond \bar x_{-\al} &=
		\sum_{n=0}^3 [\cdot, X_{-\al}]^n(\tilde x_\al)\cdot \varphi_n(H+n)\cdot [X_\al,\cdot]^n(\tilde x_{-\al})+\Romanbar{II}\\
		&= -\tilde x_{-\al} \tilde x_\al + H + \varphi_1(H+1) \tilde h^2+\varphi_2(H+1)\tilde x_{-\al}\tilde x_\al-4\varphi_3(H+1)\tilde x_{-2\al}\tilde x_{2\al}+\Romanbar{II}\\
		&=H-4\varphi_3(H+1)\bar x_{-2\al}\bDiamond\bar x_{2\al}+\big(-1+\varphi_2(H+1)\big)\big(\bar x_{-\al}\bDiamond\bar x_\al+4\varphi_1(H-1)\bar x_{-2\al}\bDiamond \bar x_{2\al}\big)\\
		&\quad+\varphi_1(H+1)\Big(\bar h\bDiamond\bar h-\varphi_1(H)\bar x_{-\al}\bDiamond\bar x_\al+\big(4\varphi_2(H)-4\varphi_1(H)\varphi_1(H-1)\big)\bar x_{-2\al}\bDiamond \bar x_{2\al}\Big)\\
		&=H+\Big(-1+\varphi_2(H+1)-\varphi_1(H+1)\varphi_1(H)\Big)\bar x_{-\al}\bDiamond\bar x_\al \\
		&\quad+\Big(-4\varphi_3(H+1)-4\varphi_1(H-1)+4\varphi_2(H+1)\varphi_1(H-1)+4\varphi_1(H+1)\varphi_2(H)\\
		&\quad-4\varphi_1(H+1)\varphi_1(H)\varphi_1(H-1)\Big)\bar x_{-2\al}\bDiamond\bar x_{2\al} + \varphi_1(H+1)\bar h\bDiamond\bar h
	\end{align*}
	using \eqref{eq:anyD2alpha},\eqref{eq:-alphaDalpha}, and \eqref{eq:hDh}. Substituting expressions for $\varphi_i(H)$, and simplifying, gives \eqref{eq:alphaD-alpha}.
	
	Relation \eqref{eq:alphaD-2alpha} follows directly from \eqref{eq:2alphaD-alpha} by applying the anti-automorphism $\Theta$ to both sides. Similarly \eqref{eq:hD-alpha} follows from \eqref{eq:alphaDh}, and \eqref{eq:hD-2alpha} follows from \eqref{eq:2alphaDh}.
\end{proof}

\section*{Appendix B}
\begin{proof}[Details of the proof for Lemma \eqref{lem:QuadAnsatz}]
%
The following expressions hold in the double coset algebra ($U/\Romanbar{II}, \bDiamond)$, where we suppress any adorning bars of basis elements and $\bDiamond$ in products. We also write $\hat h=(H-1)h$ for simplicity.
Thus, we write $C^{(2)}$ as
\begin{equation}
C^{(2)} = f_2(H) x_{-2\al} x_{2\al} + f_1(H)x_{-\al}x_{\al} + f_0(H)\hat h\hat h+g(H).
\end{equation}

We have
\begin{align*}
C^{(2)}x_{-\al} &=f_2(H)x_{-2\al}(x_{2\al}x_{-\al})+f_1(H)x_{-\al}(x_\al x_{-\al}) + f_0(H)x_{-\al}\hat h\hat h + g(H)x_{-\al}\\
&\overset{\eqref{eq:2alphaD-alpha},\eqref{eq:alphaD-alpha}}{=}f_2(H)x_{-2\al}\Big(\big(1-\frac{2}{H(H-1)}\big)x_{-\al}x_{2\al}+\frac{2}{H+1}hx_\al\Big)\\
&\quad+ f_1(H)x_{-\al}\Big( -\big(1+\frac{1}{H-1}\big)x_{-\al}x_\al+\frac{4H}{(H-1)(H-2)}x_{-2\al}x_{2\al}-\frac{1}{H(H-1)^2}\hat h\hat h+H\Big)\\
&\quad+ f_0(H)x_{-\al}\hat h\hat h+g(H)x_{-\al}\\
&=f_2(H)\big(1-\frac{2}{(H-2)(H-3)}\big)x_{-2\al}x_{-\al}x_{2\al}\\
&\quad+ f_2(H)\frac{2}{H-1}x_{-2\al}hx_\al\\
&\quad- f_1(H)\big(1+\frac{1}{H-2}\big)x_{-\al}x_{-\al}x_\al\\
&\quad+ f_1(H)\frac{4(H-1)}{(H-2)(H-3)}\big(1-\frac{2}{H-2}\big)x_{-2\al}x_{-\al}x_{2\al}\\
&\quad+f_0(H)x_{-\al}\hat h\hat h-f_1(H)\frac{1}{(H-1)(H-2)^2}x_{-\al}\hat h\hat h\\
&\quad+g(H)x_{-\al}+f_1(H)(H-1)x_{-\al}\\
&\overset{\eqref{eq:-alphaD-alpha-rel},\eqref{eq:-alphaD-2alpha}}{=} \Big(f_2(H)\big(1-\frac{2}{(H-2)(H-3)}\big)+f_1(H)\frac{4(H-1)}{(H-2)(H-3)}\big(1-\frac{2}{H-2}\big)\Big)x_{-2\al}x_{-\al}x_{2\al}\\
&\quad+\Big(f_2(H)\frac{2}{H-1}+f_1(H)\big(1+\frac{1}{H-2}\big)\frac{2}{H-2}\Big)x_{-2\al}hx_\al\\
&\quad+\Big(f_0(H)-f_1(H)\frac{1}{(H-1)(H-2)^2}\Big)x_{-\al}\hat h\hat h\\
&\quad+\Big(g(H)+f_1(H)(H-1)\Big)x_{-\al}.
\end{align*}

On the other hand,
\begin{align*}
x_{-\al}C^{(2)} &= f_2(H-1)(x_{-\al}x_{-2\al})x_{2\al}+f_1(H-1)(x_{-\al}x_{-\al})x_\al+f_0(H-1)x_{-\al}\hat h\hat h+g_0(H-1)x_{-\al}\\
&\overset{\eqref{eq:-alphaD-2alpha},\eqref{eq:-alphaD-alpha-rel}}{=} f_2(H-1)\Big(1-\frac{2}{H-2}\Big)x_{-2\al}x_{-\al}x_{2\al}\\
&\quad-f_1(H-1)\frac{2}{H-2}x_{-2\al}hx_\al\\
&\quad+f_0(H-1)x_{-\al}\hat h\hat h\\
&\quad+g(H-1)x_{-\al}.
\end{align*}
By the PBW theorem, Proposition \ref{prp:PBW}, $C^{(2)}$ anti-commutes with $x_{-\al}$ if and only if
\begin{align*}
f_1(H-1)\frac{2}{H-2}&=f_2(H)\frac{2}{H-1}+f_1(H)\big(1+\frac{1}{H-2}\big)\frac{2}{H-2}\\
-f_2(H-1)\big(1-\frac{2}{H-2}\big)&=f_2(H)\big(1-\frac{2}{(H-2)(H-3)}\big)+f_1(H)\frac{4(H-1)}{(H-2)(H-3)}\big(1-\frac{2}{H-2}\big) \\
-f_0(H-1)&=f_0(H)-f_1(H)\frac{1}{(H-1)(H-2)^2}\\
-g(H-1)&=g(H)+f_1(H)(H-1)
\end{align*}
Simplifying these equations gives \eqref{eq:f1},\eqref{eq:f2},\eqref{eq:f0},\eqref{eq:g}.

\end{proof}

\printbibliography

\end{document}